\documentclass[11pt]{amsart}

\usepackage{amsmath,amsthm, amssymb}
\usepackage{pgf,tikz}
\usetikzlibrary{decorations.pathreplacing}
\usetikzlibrary{arrows}
\usepackage{graphicx}
\usepackage[capitalise]{cleveref}
\usepackage{array}
\usepackage{epstopdf}
\usepackage{xcolor}

\newcommand{\snort}{\textsc{Snort }}
\newcommand{\col}{\textsc{Col }}
\newcommand{\nogo}{\textsc{NoGo }}

\newcommand{\domineering}{\textsc{Domineering }}
\newcommand{\Snort}{\textsc{Snort}}
\newcommand{\Col}{\textsc{Col}}
\newcommand{\Nogo}{\textsc{NoGo}}

\newcommand{\Domineering}{\textsc{Domineering}}
\newcommand{\facet}[1]{\mathcal{F}(#1)}
\newcommand{\SR}[1]{\mathcal{N}(#1)}
\newcommand{\SRideal}[1]{\mathcal{N}(\Delta_{#1})}
\newcommand{\facetI}[1]{\mathcal{F}(\Delta_{#1})}
\newcommand{\SRsc}[1]{\mathcal{N}(I_{#1})}
\newcommand{\facetsc}[1]{\mathcal{F}(I_{#1})}

\newcommand{\legalI}[1]{\mathcal{L}_{#1}}
\newcommand{\illegalI}[1]{\mathcal{ILL}_{#1}}
\newcommand{\illegalcomp}[1]{\Gamma_{#1}}

\renewcommand{\L}{\mathfrak{L}}
\newcommand{\R}{\mathfrak{R}}

\theoremstyle{definition} \newtheorem{definition}{Definition}[section]
\theoremstyle{plain} \newtheorem{theorem}[definition]{Theorem}
\theoremstyle{plain} \newtheorem{corollary}[definition]{Corollary}
\theoremstyle{plain} \newtheorem{proposition}[definition]{Proposition}
\theoremstyle{plain} \newtheorem{lemma}[definition]{Lemma}
\theoremstyle{plain} 
\theoremstyle{definition} \newtheorem{example}[definition]{Example}
\theoremstyle{remark} \newtheorem{remark}[definition]{Remark}
\theoremstyle{definition} \newtheorem{construction}[definition]{Construction}

\usepackage{xspace}
\makeatletter
\DeclareRobustCommand\onedot{\futurelet\@let@token\@onedot}
\def\@onedot{\ifx\@let@token.\else.\null\fi\xspace}
 
\def\ie{{i.e}\onedot}

\def\etal{{et al}\onedot}
\makeatother

\date{}

\begin{document}
\title{Simplicial Complexes are Game Complexes}
\author{Sara Faridi}
\address{Dalhousie University\\
Halifax, Nova Scotia\\
B3H 3J5, Canada}

\author{Svenja Huntemann}
\email{svenja.huntemann@dal.ca}

\author{Richard J. Nowakowski}


\thanks{The authors' research was supported in part by the Natural Sciences and Engineering Research Council of Canada. The second author's research was also supported by the Killam Trust. Authors Faridi and Huntemann are grateful to the math department of the Technische Universit\"at Darmstadt for their hospitality while working on this project.}

\begin{abstract}
Strong placement games (SP-games) are a class of combinatorial games whose structure allows one to describe the game via simplicial complexes. A natural question is whether well-known parameters of combinatorial games, such as ``game value'', appear as invariants of the simplicial complexes. This paper is the first step in that direction. We show that every simplicial complex encodes a certain type of SP-game (called an ``invariant SP-game'') whose ruleset is independent of the board it is played on. We also show that in the class of SP-games isomorphic simplicial complexes correspond to isomorphic game trees, and hence equal game values. We also study a subclass of SP-games corresponding to flag complexes, showing that there is always a game whose corresponding complex is a flag complex no matter which board it is played on.
\end{abstract}
\maketitle

\section{Introduction}
The purpose of this paper is to unravel some of the algebraic structure underlying combinatorial games. We show that each simplicial complex is the legal complex of some invariant strong placement game (iSP-game) and board. One implication is that in most situations when studying strong placement games (SP-games) it is enough to consider those with invariance. These results will for example make it easier to study whether each game value under normal play can be achieved by an SP-game, which would affect the study of combinatorial games in general.

In \cite{FHN14a} we initiated the idea of using simplicial complexes to algebraically describe SP-games, a class of combinatorial games. To each SP-game we can assign two simplicial complexes, one representing all legal positions, the so called \textbf{legal complex}, and one representing the minimal illegal positions, the \textbf{illegal complex}. One of the main questions is what complexes appear as game complexes. In Proposition \ref{thm:legalillegal} we show that every simplicial complex is both a legal and an illegal complex of some SP-game and board. The rulesets of these  games can be quite complex though and depend highly on the board on which the game is being played. Thus we introduce \textbf{invariance} for SP-games, which, in a sense, forces rulesets to be uniform. Invariance is a concept that was introduced for subtraction games (see for example \cite{DucheneRigo10}, \cite{Larsson12a}, \cite{LHF11}), where it is defined slightly differently due to the different class of games, but has the same intent, namely that the ruleset does not depend on the board. Similar to the previous question, we are interested in which simplicial complexes come from invariant SP-games. \cref{lem:illegal} shows that every simplicial complex without an isolated vertex is the illegal complex of some iSP-game, and also that every simplicial complex is the legal complex of an iSP-game. The constructions given in all cases prove the stronger result that such SP-games exist given \textit{any} bipartition of the vertices of the simplicial complex (see \cref{thm:Gammagameandgraph,thm:legalinv}) into Left and Right positions. This construction then allows us to show that for every SP-game there exists an iSP-game such that their game trees are isomorphic. This in turn implies that their game values are the same under both normal and mis\`ere winning conditions. Thus  it is enough to only consider iSP-games in most situations. 
 
Finally, we restrict to independence games, those games for which the ruleset played on any board gives an illegal complex which is a graph. This class includes many games actually played, such as \Snort, \Col, and \Domineering, but not \textsc{NoGo}. We show that any SP-game whose illegal complex is a graph is literally equal to an invariant independence game. \\
 
In the next two subsections, we give the background in combinatorial game theory and algebra needed for the paper. Please see any of \cite{ANW07,BCG04,Siegel13} for further information in combinatorial game theory and \cite{BH93,HH11} for the algebra involved. We then show that each simplicial complex is a game complex, and finally in \cref{sec:iSP} we consider invariant SP-games, and independence games in \cref{sec:independence}.

\subsection{Combinatorial Game Theory}

A \textbf{combinatorial game} is a 2-player game with perfect information and no chance, where the two players are \textbf{Left} and \textbf{Right} (denoted by $L$ and $R$ respectively) and they do not move simultaneously. For the purposes of this paper, the winning condition is irrelevant as long as it does not contradict the other conditions for the games. 

We denote a combinatorial game by its name in \textsc{Small Caps}.

In this paper, a \textbf{board} will be a finite graph. The \textbf{pieces}, which can be thought of as tokens or as subgraphs of the board, will be placed on a non-empty collection of vertices --- exactly how is given by the \textbf{ruleset}. For a game $G$ consisting of a ruleset $R$ played on a board $B$ we will use the notation $G=(R,B)$. A \textbf{position} is a configuration of pieces on the board. A position that can be reached through a sequence of legal moves is called a \textbf{legal position}, otherwise we call it an \textbf{illegal position}. A \textbf{basic position} is a board with only one piece placed.

Given a game $G$, the \textbf{game tree} of $G$ is a directed graph, which is a tree, with the edges labelled $L$ or $R$. The vertices of the tree correspond to positions and $X\stackrel{L}{\rightarrow} Y$ if there is a legal move for Left from position $X$ to $Y$. Similarly edges labelled with an $R$ represent moves by Right. The games we consider all have finite game trees. Two games whose game trees have isomorphic structure are called \textbf{literally equal}. Note that two games that are literally equal will be equal under any winning condition.

Brown \etal \cite{BCHMMNS14} introduced a subclass of combinatorial games, which they called placement games. Their conditions are slightly weaker than what is required for this work, and we thus call our games ``strong placement games''. 

\begin{definition}\label{def:placementgame}
A \textbf{strong placement game (SP-game)} is a combinatorial game which satisfies the following:
\begin{itemize}
\item[(i)] The board is empty at the beginning of the game.
\item[(ii)] Players place pieces on empty spaces of the board according to the rules.
\item[(iii)] Pieces are not moved or removed once placed.
\item[(iv)] The rules are such that if it is possible to reach a position through a sequence of legal moves, then any sequence of moves leading to this position consists  of legal moves.
\end{itemize}
\end{definition}

Note that condition (iv) in the above definition is necessary for each position to be independent of the order of moves, which results in commutativity when representing positions by monomials (see \cref{sec:gamecomplexes}) and for the hypergraphs representing the game being simplicial complexes.

This condition also implies that any position, whether legal or illegal, in an SP-game can be decomposed into basic positions.

The rulesets in \cref{def:games} together with a board are examples of SP-games and will be used throughout the document. The first two have been introduced early in the development of game theory (see \cite{BCG04}), but, surprisingly, not much is known about them. 

\begin{example}\label{def:games}
In \Snort, players place a piece on a single vertex which is not adjacent to a vertex containing a piece from their opponent.

In \Col, players place a piece on a single vertex which is not adjacent to a vertex containing one of their own pieces.

In \nogo (see \cite{ChouTY11}), players place a piece on a single unoccupied vertex. At every point in the game, for each maximal group of connected vertices of the board that contain pieces placed by the same player, at least one of these needs to be adjacent to an empty vertex.

In \domineering (see \cite{Berlekamp88} and \cite{LMR02}), which is played on grids, both players place dominoes. Left may only place vertically, and Right only horizontally. The vertices of the board are the squares of the grid, and each piece occupies two vertices. 
 
\end{example}

Other examples of SP-games are \textsc{Node-Kayles} and \textsc{Arc-Kayles} (see for example \cite{Bodlaender93}, \cite{FT06}, \cite{Schaefer78}).

Many combinatorial games, especially SP-games, have a natural tendency to break up into smaller, independent components as play progresses. For example, after several moves the empty spaces could be split into many disconnected components and a player, on their move, then has to choose a component to move in. From this, we define a sum on games as follows: The \textbf{disjunctive sum} {\boldmath $G_1+G_2$} of two games $G_1$ and $G_2$ is the game in which at each step the current player can decide to move in either game, but not both.

Depending on a fixed winning condition, combinatorial games can be divided into equivalence classes. 
The equivalence class a game belongs to is called its \textbf{game value}. Game values form a partially ordered semi-group under disjunctive sum. When a game consists of the disjunctive sum of two sub-games on different boards, then to find the game value it is sufficient to calculate the game values of the summands and taking advantage of the additive structure. This is a very useful concept in combinatorial game theory. For more details see \cite{Siegel13}.

\subsection{Combinatorial Commutative Algebra}

Simplicial complexes are one of the main constructs we use to study SP-games. We begin by introducing the required concepts.

\begin{definition}
An (abstract) \textbf{simplicial complex} $\Delta$ on a finite vertex set $V$ is a set of subsets, called \textbf{faces}, of $V$ with the conditions that if $A\in \Delta$ and $B\subseteq A$, then $B\in \Delta$. The \textbf{facets} of a simplicial complex $\Delta$ are the maximal faces of $\Delta$ with respect to inclusion. A \textbf{non-face} of a simplicial complex $\Delta$ is a subset of its vertices that is not a face. 
\end{definition}

Note that a simplicial complex with a fixed vertex set is uniquely determined by its facets. Thus a simplicial complex $\Delta$ with facets $F_1,\ldots, F_k$ is denoted by $\Delta=\langle F_1,\ldots, F_k\rangle$. The vertex set of $\Delta$ is also denoted as $V(\Delta)$.

A simplicial complex of the form $\Delta=\langle \{i_1,i_2,\ldots, i_r\}\rangle$, where $V(\Delta)=\{i_1,i_2,\ldots, i_r\}$, is called a \textbf{simplex}.

\begin{definition}
Given a face $F$ of a simplicial complex $\Delta$, its \textbf{dimension} $\dim(F)$ is $|F|-1$. The dimension of the simplicial complex $\Delta$ is the maximum dimension of any of its faces. A simplicial complex $\Delta$ is called \textbf{pure} if all its facets are of the same dimension. The $\mathbf{k}$\textbf{-skeleton} $\Delta^{[k]}$ of a simplicial complex $\Delta$ is the simplicial complex whose facets are the $k$-dimensional faces of $\Delta$.
\end{definition}

The other structures used to study SP-games are square-free monomial ideals, which we introduce now.

\begin{definition}
Let $k$ be a field and $S$ the polynomial ring $k[x_1,\ldots,x_n]$. A product $x_{1}^{a_{1}}\ldots x_{n}^{a_{n}}\in S$, where the $a_i$ are non-negative integers, is called a \textbf{monomial}. Such a monomial is called \textbf{square-free} if each $a_i$ is either 0 or 1.
\end{definition}
\begin{definition}
Let $k$ be a field and $S$ the polynomial ring $k[x_1,\ldots,x_n]$. A \textbf{monomial ideal} of $S$ is an ideal generated by monomials in $S$. A monomial ideal is called a \textbf{square-free monomial ideal} if it is generated by square-free monomials.
\end{definition}

Let $k$ be a field and $S=k[x_1,\ldots,x_n]$ a polynomial ring. There is a one-to-one correspondence between subsets $\{i_1,\ldots, i_r\}$ of $[n]$ and square-free monomials $x_{i_1}\cdots x_{i_r}$ of $S$. Using this observation we can associate to a square-free monomial ideal two unique simplicial complexes: the facet complex and the Stanley-Reisner complex.

\begin{definition}\label{def:SRFcomplex}
The \textbf{facet complex} of a square-free monomial ideal $I$ of $S$, denoted by $\facetsc{}$, is the simplicial complex whose facets correspond to the square-free monomials in the minimal generating set of $I$. The \textbf{Stanley-Reisner complex} of a square-free monomial ideal $I$ of $S$, denoted by $\SRsc{}$, is the simplicial complex whose faces correspond to the square-free monomials not in $I$. In other words,
\begin{align*}
\facetsc{}&=\langle \{i_1,\ldots, i_r\}\mid x_{i_1}\cdots x_{i_r}\text{ minimal generator of }I\rangle\text{ and}\\
\SRideal{}&=\langle \{i_1, \ldots, i_r\}\mid x_{i_1}\cdots x_{i_r}\not\in I\rangle.
\end{align*}
\end{definition}

This correspondence works in the opposite direction as well.

\begin{definition}\label{def:SRFideal}
The \textbf{facet ideal} of a simplicial complex $\Delta$, denoted by $\facetI{}$, is the ideal of $S$ generated by the monomials corresponding to the facets of $\Delta$. The \textbf{Stanley-Reisner ideal} of a simplicial complex $\Delta$, denoted by $\SRideal{}$, is the ideal of $S$ generated by the monomials corresponding to the minimal non-faces of $\Delta$. In other words,
\begin{align*}
\facetI{}&=\left(x_{i_1}\cdots x_{i_r}\mid \{i_1,\ldots,i_r\}\text{ facet of }\Delta\right)\text{ and}\\
\SRideal{}&=\left(x_{i_1}\cdots x_{i_r}\mid \{i_1,\ldots,i_r\}\not\in\Delta\right).
\end{align*}
\end{definition}

Note that although it is common that for a given simplicial complex $\Delta$ the vertex set simply consists of the 0-dimensional faces of $\Delta$, this is not always the case. Due to the difference in underlying rings, successively applying the Stanley-Reisner and facet operators would result in a different simplicial complex.

\subsection{Game Complexes and Ideals}\label{sec:gamecomplexes}

We now introduce the construction of simplicial complexes and square-free monomial ideals which are related to SP-games. Unless otherwise specified, let the underlying ring be $S=k[x_1,\ldots,x_m,y_1,\ldots, y_n]$, where $k$ is a field, $m$ the number of basic positions with a Left piece, and $n$ the number of basic positions with a Right piece.

A square-free monomial $z$ of $S$ represents a position $P$ in the game if it is the product over those $x_i$ and $y_j$ such that Left has played in the basic position $i$ and Right has played in the basic position $j$ in order to reach $P$. By condition (iv) in \cref{def:placementgame}, the order of moves to reach $P$ does not matter, thus we have commutativity.

\begin{example}\label{ex:O12basicpositions}
Consider the game \textsc{O12}, in which Left claims a single vertex, and Right two adjacent vertices, played on $P_4$. We number the vertices consecutively from one end as 1, 2, 3, 4. The Left basic position $i$ is the position in which Left has played on vertex $i$, and the Right basic position $j$ is the position in which Right has played on vertices $j$ and $j+1$. Since Left has 4 basic positions, and Right has 3, the underlying ring is $S=k[x_1,x_2,x_3,x_4,y_1,y_2,y_3]$. The position
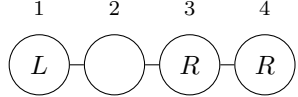
\begin{center}
\begin{tikzpicture}[vertex/.style={circle, draw, minimum size=8mm, font=\small}]
	\node[vertex] (1) at (0,0) {$L$};
	\node[vertex] (2) at (1,0) {};
	\node[vertex] (3) at (2,0) {$R$};
	\node[vertex] (4) at (3,0) {$R$};
	\draw (1)--(2)--(3)--(4);
	\begin{scriptsize}
	\node at (0,0.75) {1};
	\node at (1,0.75) {2};
	\node at (2,0.75) {3};
	\node at (3,0.75) {4};
	\end{scriptsize}
\end{tikzpicture}
\end{center}
is represented by the monomial $x_1y_3$.
\end{example}

A legal position is called a \textbf{maximal legal position} if placing any further piece is illegal, \ie it is not properly contained in any other legal position.

If we sort the monomials representing illegal positions by divisibility, the positions corresponding to the minimal elements are called \textbf{minimal illegal positions}. Equivalently, an illegal position is a minimal illegal position if any proper subset of the pieces placed forms a legal position.

\begin{definition}\cite{FHN14a}\label{def:gameComplexIdeal}
If $(R,B)$ is an SP-game, then
\begin{itemize}
\item The \textbf{legal ideal}, $\legalI{R,B}$, is the ideal of $S$ generated by the monomials representing maximal legal positions.
\item The \textbf{illegal ideal}, $\illegalI{R,B}$, is the ideal generated by the monomials representing minimal illegal positions.
\item The \textbf{legal complex}, $\Delta_{R,B}$, is the facet complex of the legal ideal.
\item The \textbf{illegal complex}, $\Gamma_{R,B}$, is the facet complex of the illegal ideal.
\end{itemize}
\end{definition}

Some of the results we discuss in this paper hold for both the legal and illegal complex of some game and board. For brevity, we will use the term \textbf{game complex} when discussing a simplicial complex which is either a legal or illegal complex.

Note that condition (iv) in \cref{def:placementgame} implies that the order of moves does not matter, which gives us commutativity when representing positions by monomials. Thus the legal and illegal ideal are indeed commutative ideals. The condition also implies that given any legal position, any subset of the pieces played gives a legal position as well, and thus the hypergraphs representing the game are indeed simplicial complexes.

The next example demonstrates these concepts. This also illustrates again that the vertices of the complexes are the basic positions, not the vertices of the board. 
\begin{example}
Consider \textsc{O12} played on $P_3$. Similar to \cref{ex:O12basicpositions} the underlying ring is $S=k[x_1,x_2,x_3,y_1,y_2]$.

The maximal legal positions are represented by the monomials $x_1x_2x_3$, $x_1y_2$, and $x_3y_1$. Thus we have the legal ideal \[\legalI{\textsc{O12}, P_3}=(x_1x_2x_3,x_1y_2, x_3y_1).\]  The legal complex is given in \cref{fig:domex2}.

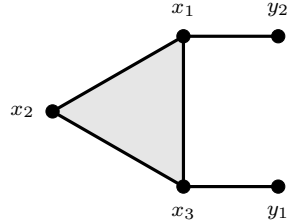
\begin{figure}[!ht]
\begin{center}
\begin{tikzpicture}[line cap=round,line join=round,>=triangle 45,x=1.0cm,y=1.0cm]
\filldraw[line width=1.2pt,fill=black,fill opacity=0.1] (0.,1.) -- (1.74,0.) -- (1.74,2) -- cycle;
\draw [line width=1.2pt] (1.74,0.)-- (3.,0.);
\draw [line width=1.2pt] (1.74,2)-- (3.,2.);
\begin{scriptsize}
\draw [fill=black] (0.,1.) circle (2.5pt);
\draw[color=black] (-0.4,1) node {$x_2$};
\draw [fill=black] (1.74,0.) circle (2.5pt);
\draw[color=black] (1.74,-0.35) node {$x_3$};
\draw [fill=black] (1.74,2) circle (2.5pt);
\draw[color=black] (1.74,2.35) node {$x_1$};
\draw [fill=black] (3.,0.) circle (2.5pt);
\draw[color=black] (3,-0.35) node {$y_1$};
\draw [fill=black] (3.,2.) circle (2.5pt);
\draw[color=black] (3,2.35) node {$y_2$};
\end{scriptsize}
\end{tikzpicture}
\end{center}
\caption{The legal complex $\Delta_{\textsc{O12}, P_3}$}
\label{fig:domex2}
\end{figure}

The minimal illegal positions are represented by the monomials $x_1y_1$, $x_2y_1$, $x_2y_2$, $x_3y_2$, and $y_1y_2$. Thus we have the illegal ideal
\[\illegalI{\textsc{O12}, P_3}=(x_1y_1,x_2y_1, x_2y_2, x_3y_2, y_1y_2).\] The illegal complex is given in \cref{fig:domex3}.

\begin{figure}[!ht]
\begin{center}
\begin{tikzpicture}[line cap=round,line join=round,>=triangle 45,x=1.0cm,y=1.0cm]
\draw [line width=1.2pt] (1.,1.)-- (2.,-1.);
\draw [line width=1.2pt] (3.,1.)-- (2.,-1.);
\draw [line width=1.2pt] (3.,1.)-- (4.,-1.);
\draw [line width=1.2pt] (5.,1.)-- (4.,-1.);
\draw [line width=1.2pt] (2.,-1.)-- (4.,-1.);
\begin{scriptsize}
\draw [fill=black] (1.,1.) circle (2.5pt);
\draw[color=black] (1,1.35) node {$x_1$};
\draw [fill=black] (2.,-1.) circle (2.5pt);
\draw[color=black] (2,-1.35) node {$y_1$};
\draw [fill=black] (3.,1.) circle (2.5pt);
\draw[color=black] (3,1.35) node {$x_2$};
\draw [fill=black] (4.,-1.) circle (2.5pt);
\draw[color=black] (4,-1.35) node {$y_2$};
\draw [fill=black] (5.,1.) circle (2.5pt);
\draw[color=black] (5,1.35) node {$x_3$};
\end{scriptsize}
\end{tikzpicture}
\end{center}
\caption{The illegal complex $\Gamma_{\textsc{O12}, P_3}$}
\label{fig:domex3}
\end{figure}
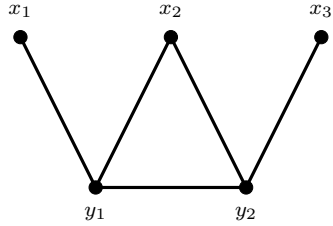
\end{example}

This example also illustrates the following result. See \cite{FHN14a} for more details.

\begin{remark}
Note that the faces of the legal complex $\Delta_{R,B}$ represent the legal positions of $(R,B)$, while the facets of $\Gamma_{R,B}$ represent the minimal illegal positions. In short we have
\begin{itemize}
	\item[(1)] $\legalI{R,B}=\facetI{R,B},$
	\item[(2)] $\illegalI{R,B}=\facet{\illegalcomp{R,B}}=\SRideal{R,B},$
\end{itemize}
or equivalently
\begin{itemize}
	\item[(1)] $\Delta_{R,B}=\facet{\legalI{R,B}}=\SR{\illegalI{R,B}},$
	\item[(2)] $\Gamma_{R,B}=\facet{\illegalI{R,B}}.$
\end{itemize}
This will be used throughout this paper.
\end{remark}

This relationship in particular also gives that $\Gamma_{R,B}=\facet{\SR{\Delta_{R,B}}}$, allowing us to move from the legal complex to the illegal and vice versa using the facet and Stanley-Reisner correspondences. Note though that this relationship only holds if we let the vertex set of the game complexes be all basic positions rather than just the 0-dimensional faces.

\begin{example}
Consider the SP-ruleset $R$ in which both players claim a single vertex, but Right may not place on vertex of degree 1, while Left has no restrictions for placement. The legal complex for playing on $B=P_3$ is then $\langle x_1x_2x_3,x_1y_2x_3\rangle$ with the underlying ring being $S=k[x_1,x_2,x_3,y_1,y_2,y_3]$. The illegal complex is $\Gamma=\langle y_1,y_3,x_2y_2\rangle$.

If we would let the underlying ring for $\Gamma$ be $S'=k[x_2,y_1,y_2,y_3]$, then $\SR{\facet{\Gamma}}=\langle x_2,y_2\rangle$. This is the legal complex of $R$ played on a single vertex, which is a very different game from $R$ played on $B$ as Left no longer has the two guaranteed moves at the ends.
\end{example}

Essentially, for the legal complex removing elements from the vertex set which are not 0-dimensional faces corresponds to removing illegal basic positions, or equivalently isolated vertices in the illegal complex, and does not change the game in any significant way. As seen in the example above, the same is not true though for the vertex set of the illegal complex: removing elements from the vertex set of the illegal complex which do not appear in any facets would imply removing a basic position which is always legal, thus changing the game significantly. For generality, we will keep letting the vertex set be the set of basic positions.

It is important to note that the legal and illegal complexes and corresponding ideals have an extra layer of structure. The monomials have elements $\{x_1,x_2,\ldots,x_m\}$ and  $\{y_1,y_2,\ldots,y_n\}$ and the complexes have their elements partitioned into those corresponding to the Left and Right basic positions. Thus, when showing a simplicial complex is equal to a game complex, we must also specify the partition.

In general, we call a simplicial complex whose vertex set is bipartitioned into sets $\L$ and $\R$ an $(\L,\R)$-labelled simplicial complex.

We occasionally also talk about isomorphic boards, with which we mean the boards are isomorphic as graphs and contain the same pieces. For SP-games we formally define a board isomorphism as follows.
\begin{definition}
Let $B_1$ and $B_2$ be two boards, potentially not empty. A map $\phi:B_1\rightarrow B_2$ is a \textbf{board isomorphism} if
\begin{enumerate}
\item $\phi$ is a graph isomorphism, that is a bijection of the vertex sets of $B_1$ and $B_2$ such that $\{v_1,v_2\}$ is an edge of $B_1$ if and only if $\{\phi(v_1),\phi(v_2)\}$ is an edge of $B_2$;
\item The vertex $v$ of $B_1$ contains a Left piece if and only if the vertex $\phi(v)$ of $B_2$ contains a Left piece; and
\item The vertex $w$ of $B_1$ contains a Right piece if and only if the vertex $\phi(w)$ of $B_2$ contains a Right piece.
\end{enumerate}
\end{definition}

The following proposition shows that two games with isomorphic legal complexes have isomorphic game trees, and as a consequence the same game value under most winning conditions (such as normal play and mis\`ere, see \cite{Siegel13}). Thus using simplicial complexes helps us to identify when two games are literally equal.

\begin{proposition}[\textbf{Isomorphic Game Trees of SP-Games}]\label{thm:isomorphicgametrees}
If two SP-games $(R_1,B_1)$ and $(R_2,B_2)$ have isomorphic legal complexes, then their game trees are isomorphic, \ie they are literally equal.
\end{proposition}
\begin{proof} We prove that isomorphic legal complexes imply isomorphic game trees by induction on the size of the faces (\ie the number of pieces in a position). The empty face (\ie empty board) corresponds to the root of the game tree, thus is trivially the same for both games.

Now assume that the game trees are isomorphic up to positions with $k$ pieces played.

Consider a position $P_1$ in the game $G_1$ played on $B_1$ with $k$ pieces played. Let $F_1$ be the face of $\Delta_{G_1,B_1}$ (of dimension $k-1$) corresponding to $P_1$. Since $\Delta_{G_1,B_1}$ and $\Delta_{G_2,B_2}$ are isomorphic, there exists a face $F_2\in \Delta_{G_2,B_2}$ (of dimension $k-1$) isomorphic to $F_1$, corresponding to a position $P_2$ of $G_2$, which also has $k$ pieces placed.

Now let $P_1'$ be any option of $P_1$ and $F_1'$ be the corresponding face in $\Delta_{G_1,B_1}$. Then there exists a vertex $v$ such that $F_1'=F_1\cup\{v\}$. Let $F_2'$ be the face of $\Delta_{G_2,B_2}$ isomorphic to $F_1'$. Then there exists a vertex $w$ (corresponding to $v$) such that $F_2'=F_2\cup\{w\}$. Thus the position $P_2'$ corresponding to $F_2'$ is an option of $P_2$.

Further, since the legal complexes have the same bipartition, we have that the following are equivalent:
\begin{enumerate}
\item The position $P_1'$ is a Left- (Right-)option of $P_1$.
\item The vertex $v$ belongs to $\L$ ($\R$).
\item The vertex $w$ belongs to $\L$ ($\R$).
\item The position $P_2'$ is a Left- (Right-)option of $P_2$.
\end{enumerate}

Thus for any option of $P_1$ there exists an option of $P_2$ and vice-versa, which shows that the game trees of $G_1,B_1$ and $G_2,B_2$ are isomorphic up to positions of $k+1$, and by induction they are entirely isomorphic.
\end{proof}

Note though that the converse of \cref{thm:isomorphicgametrees} is not true, as the following example demonstrates. We are grateful to Alex Fink for providing this example.
\begin{example}[Alex Fink]\label{ex:treeNotImpliesLegal}
Consider a ruleset $R$ in which all pieces occupy a single vertex, have to be adjacent to all previously placed ones, at most two pieces may be placed, and only Left may play. Then $\Delta_{R,B}\cong B$ for all boards $B$. In particular, consider $B_1$ being a disjoint union of two 3-cycles, and $B_2$ being a 6-cycle with labels for basic positions as below.

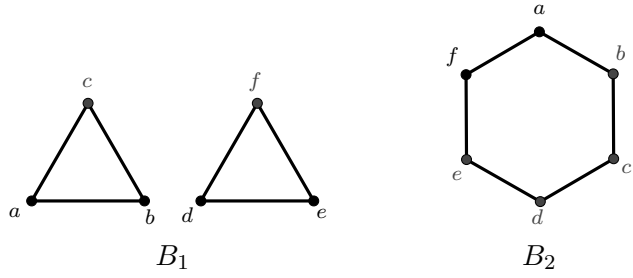
\begin{center}
\definecolor{uuuuuu}{rgb}{0.26666666666666666,0.26666666666666666,0.26666666666666666}
\begin{tikzpicture}[line cap=round,line join=round,>=triangle 45,x=1.0cm,y=1.0cm,scale=0.75]
\draw [line width=1.2pt] (-3.,1.)-- (-1.,1.);
\draw [line width=1.2pt] (-1.,1.)-- (-2.,2.7320508075688776);
\draw [line width=1.2pt] (-2.,2.7320508075688776)-- (-3.,1.);
\draw [line width=1.2pt] (0.,1.)-- (2.,1.);
\draw [line width=1.2pt] (2.,1.)-- (1.,2.7320508075688776);
\draw [line width=1.2pt] (1.,2.7320508075688776)-- (0.,1.);
\draw [line width=1.2pt] (6.,4.)-- (4.7,3.24);
\draw [line width=1.2pt] (4.7,3.24)-- (4.708179306876172,1.7341669750802309);
\draw [line width=1.2pt] (4.708179306876172,1.7341669750802309)-- (6.016358613752345,0.9883339501604596);
\draw [line width=1.2pt] (6.016358613752345,0.9883339501604596)-- (7.316358613752345,1.7483339501604582);
\draw [line width=1.2pt] (7.316358613752345,1.7483339501604582)-- (7.308179306876173,3.2541669750802282);
\draw [line width=1.2pt] (7.308179306876173,3.2541669750802282)-- (6.,4.);
\draw (-0.5,0) node {$B_1$};
\draw (6,0) node {$B_2$};
\begin{scriptsize}
\draw [fill=black] (-3.,1.) circle (2.5pt);
\draw[color=black] (-3.3,0.79) node {$a$};
\draw [fill=black] (-1.,1.) circle (2.5pt);
\draw[color=black] (-0.9,0.73) node {$b$};
\draw [fill=black] (0.,1.) circle (2.5pt);
\draw[color=black] (-0.24,0.75) node {$d$};
\draw [fill=black] (2.,1.) circle (2.5pt);
\draw[color=black] (2.14,0.77) node {$e$};
\draw [fill=black] (6.,4.) circle (2.5pt);
\draw[color=black] (6.,4.39) node {$a$};
\draw [fill=black] (4.7,3.24) circle (2.5pt);
\draw[color=black] (4.46,3.55) node {$f$};
\draw [fill=uuuuuu] (-2.,2.7320508075688776) circle (2.5pt);
\draw[color=uuuuuu] (-2.02,3.11) node {$c$};
\draw [fill=uuuuuu] (1.,2.7320508075688776) circle (2.5pt);
\draw[color=uuuuuu] (0.96,3.11) node {$f$};
\draw [fill=uuuuuu] (4.708179306876172,1.7341669750802309) circle (2.5pt);
\draw[color=uuuuuu] (4.54,1.45) node {$e$};
\draw [fill=uuuuuu] (6.016358613752345,0.9883339501604596) circle (2.5pt);
\draw[color=uuuuuu] (5.96,0.73) node {$d$};
\draw [fill=uuuuuu] (7.316358613752345,1.7483339501604582) circle (2.5pt);
\draw[color=uuuuuu] (7.54,1.57) node {$c$};
\draw [fill=uuuuuu] (7.308179306876173,3.2541669750802282) circle (2.5pt);
\draw[color=uuuuuu] (7.44,3.63) node {$b$};
\end{scriptsize}
\end{tikzpicture}
\end{center} 

The game trees for $(R,B_1)$ and $(R,B_2)$ (shown below on the top and bottom respectively) are isomorphic even though the legal complexes are not.
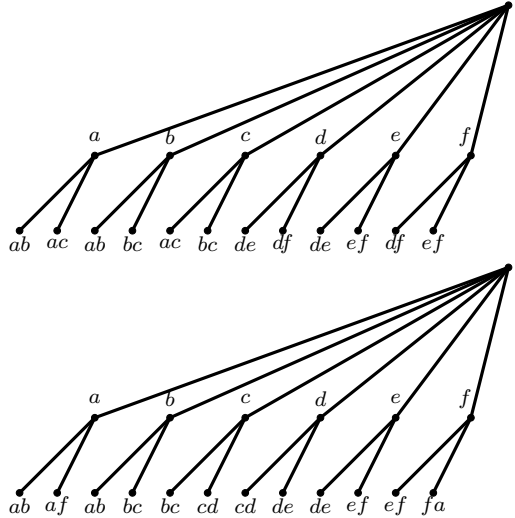
\begin{center}
\begin{tikzpicture}[line cap=round,line join=round,>=triangle 45,x=1.0cm,y=1.0cm,scale=0.5]
\draw [line width=1.2pt] (8.,5.)-- (7.,1.);
\draw [line width=1.2pt] (8.,5.)-- (5.,1.);
\draw [line width=1.2pt] (8.,5.)-- (3.,1.);
\draw [line width=1.2pt] (8.,5.)-- (1.,1.);
\draw [line width=1.2pt] (8.,5.)-- (-1.,1.);
\draw [line width=1.2pt] (8.,5.)-- (-3.,1.);
\draw [line width=1.2pt] (7.,1.)-- (6.,-1.);
\draw [line width=1.2pt] (7.,1.)-- (5.,-1.);
\draw [line width=1.2pt] (5.,1.)-- (4.,-1.);
\draw [line width=1.2pt] (5.,1.)-- (3.,-1.);
\draw [line width=1.2pt] (3.,1.)-- (2.,-1.);
\draw [line width=1.2pt] (3.,1.)-- (1.,-1.);
\draw [line width=1.2pt] (1.,1.)-- (0.,-1.);
\draw [line width=1.2pt] (1.,1.)-- (-1.,-1.);
\draw [line width=1.2pt] (-1.,1.)-- (-2.,-1.);
\draw [line width=1.2pt] (-1.,1.)-- (-3.,-1.);
\draw [line width=1.2pt] (-3.,1.)-- (-4.,-1.);
\draw [line width=1.2pt] (-3.,1.)-- (-5.,-1.);
\begin{scriptsize}
\draw [fill=black] (-3.,1.) circle (2.5pt);
\draw[color=black] (-3,1.5) node {$a$};
\draw [fill=black] (-1.,1.) circle (2.5pt);
\draw[color=black] (-1,1.5) node {$b$};
\draw [fill=black] (1.,1.) circle (2.5pt);
\draw[color=black] (1,1.5) node {$c$};
\draw [fill=black] (3.,1.) circle (2.5pt);
\draw[color=black] (3,1.5) node {$d$};
\draw [fill=black] (5.,1.) circle (2.5pt);
\draw[color=black] (5,1.5) node {$e$};
\draw [fill=black] (7.,1.) circle (2.5pt);
\draw[color=black] (6.84,1.5) node {$f$};
\draw [fill=black] (8.,5.) circle (2.5pt);
\draw [fill=black] (6.,-1.) circle (2.5pt);
\draw[color=black] (6,-1.35) node {$ef$};
\draw [fill=black] (5.,-1.) circle (2.5pt);
\draw[color=black] (5,-1.35) node {$df$};
\draw [fill=black] (4.,-1.) circle (2.5pt);
\draw[color=black] (4,-1.35) node {$ef$};
\draw [fill=black] (3.,-1.) circle (2.5pt);
\draw[color=black] (3,-1.35) node {$de$};
\draw [fill=black] (2.,-1.) circle (2.5pt);
\draw[color=black] (2,-1.35) node {$df$};
\draw [fill=black] (1.,-1.) circle (2.5pt);
\draw[color=black] (1,-1.35) node {$de$};
\draw [fill=black] (0.,-1.) circle (2.5pt);
\draw[color=black] (0,-1.35) node {$bc$};
\draw [fill=black] (-1.,-1.) circle (2.5pt);
\draw[color=black] (-1,-1.35) node {$ac$};
\draw [fill=black] (-2.,-1.) circle (2.5pt);
\draw[color=black] (-2,-1.35) node {$bc$};
\draw [fill=black] (-3.,-1.) circle (2.5pt);
\draw[color=black] (-3,-1.35) node {$ab$};
\draw [fill=black] (-4.,-1.) circle (2.5pt);
\draw[color=black] (-4,-1.35) node {$ac$};
\draw [fill=black] (-5.,-1.) circle (2.5pt);
\draw[color=black] (-5,-1.35) node {$ab$};
\end{scriptsize}
\end{tikzpicture}
\begin{tikzpicture}[line cap=round,line join=round,>=triangle 45,x=1.0cm,y=1.0cm,scale=0.5]
\draw [line width=1.2pt] (8.,5.)-- (7.,1.);
\draw [line width=1.2pt] (8.,5.)-- (5.,1.);
\draw [line width=1.2pt] (8.,5.)-- (3.,1.);
\draw [line width=1.2pt] (8.,5.)-- (1.,1.);
\draw [line width=1.2pt] (8.,5.)-- (-1.,1.);
\draw [line width=1.2pt] (8.,5.)-- (-3.,1.);
\draw [line width=1.2pt] (7.,1.)-- (6.,-1.);
\draw [line width=1.2pt] (7.,1.)-- (5.,-1.);
\draw [line width=1.2pt] (5.,1.)-- (4.,-1.);
\draw [line width=1.2pt] (5.,1.)-- (3.,-1.);
\draw [line width=1.2pt] (3.,1.)-- (2.,-1.);
\draw [line width=1.2pt] (3.,1.)-- (1.,-1.);
\draw [line width=1.2pt] (1.,1.)-- (0.,-1.);
\draw [line width=1.2pt] (1.,1.)-- (-1.,-1.);
\draw [line width=1.2pt] (-1.,1.)-- (-2.,-1.);
\draw [line width=1.2pt] (-1.,1.)-- (-3.,-1.);
\draw [line width=1.2pt] (-3.,1.)-- (-4.,-1.);
\draw [line width=1.2pt] (-3.,1.)-- (-5.,-1.);
\begin{scriptsize}
\draw [fill=black] (-3.,1.) circle (2.5pt);
\draw[color=black] (-3,1.5) node {$a$};
\draw [fill=black] (-1.,1.) circle (2.5pt);
\draw[color=black] (-1,1.5) node {$b$};
\draw [fill=black] (1.,1.) circle (2.5pt);
\draw[color=black] (1,1.5) node {$c$};
\draw [fill=black] (3.,1.) circle (2.5pt);
\draw[color=black] (3,1.5) node {$d$};
\draw [fill=black] (5.,1.) circle (2.5pt);
\draw[color=black] (5,1.5) node {$e$};
\draw [fill=black] (7.,1.) circle (2.5pt);
\draw[color=black] (6.84,1.5) node {$f$};
\draw [fill=black] (8.,5.) circle (2.5pt);
\draw [fill=black] (6.,-1.) circle (2.5pt);
\draw[color=black] (6,-1.35) node {$fa$};
\draw [fill=black] (5.,-1.) circle (2.5pt);
\draw[color=black] (5,-1.35) node {$ef$};
\draw [fill=black] (4.,-1.) circle (2.5pt);
\draw[color=black] (4,-1.35) node {$ef$};
\draw [fill=black] (3.,-1.) circle (2.5pt);
\draw[color=black] (3,-1.35) node {$de$};
\draw [fill=black] (2.,-1.) circle (2.5pt);
\draw[color=black] (2,-1.35) node {$de$};
\draw [fill=black] (1.,-1.) circle (2.5pt);
\draw[color=black] (1,-1.35) node {$cd$};
\draw [fill=black] (0.,-1.) circle (2.5pt);
\draw[color=black] (0,-1.35) node {$cd$};
\draw [fill=black] (-1.,-1.) circle (2.5pt);
\draw[color=black] (-1,-1.35) node {$bc$};
\draw [fill=black] (-2.,-1.) circle (2.5pt);
\draw[color=black] (-2,-1.35) node {$bc$};
\draw [fill=black] (-3.,-1.) circle (2.5pt);
\draw[color=black] (-3,-1.35) node {$ab$};
\draw [fill=black] (-4.,-1.) circle (2.5pt);
\draw[color=black] (-4,-1.35) node {$af$};
\draw [fill=black] (-5.,-1.) circle (2.5pt);
\draw[color=black] (-5,-1.35) node {$ab$};
\end{scriptsize}
\end{tikzpicture}
\end{center}
\end{example}

This also gives an indication that the legal complex of an SP-game is a better representative for the SP-game than the tree as it conveys more structure.

On the other hand, if the illegal complexes are isomorphic, it is not always true that the game trees are isomorphic. For example, consider the ruleset $R$ in which neither player can place on a vertex of degree 1. We then have \[\Gamma_{R,P_2}=\langle x_1,x_2,y_1,y_2\rangle\cong\Gamma_{R,P_3}=\langle x_1, x_3, y_1, y_3\rangle.\] The legal complexes $\Delta_{R,P_2}=\emptyset$ and $\Delta_{R, P_3}=\langle x_2, y_2\rangle$ are not isomorphic. And since there are no legal moves in $(R, P_2)$, but there are in $(R,P_3)$, their game trees are not isomorphic either. Another occurrence of this is if there are moves that are always playable in one game, but these moves do not occur at all in the second game. 

Finally, we are able to characterize what the legal complex of the disjunctive sum of two games looks like. Given two simplicial complexes $\Delta$ and $\Delta'$ their \textbf{join} $\Delta *\Delta'$ is the simplicial complex whose facets are all the unions of a facet of $\Delta$ with a facet of $\Delta'$.

\begin{proposition}\label{thm:disjunctiveStructure}
Let $(R,B)$ and $(R',B')$ be two SP-games with legal complexes $\Delta_{R,B}$ and $\Delta_{R',B'}$. Then
\[\Delta_{(R,B)+(R',B')}=\Delta_{R,B} *\Delta_{R',B'}\]
is the legal complex of the disjunctive sum $(R,B)+(R',B')$.
\end{proposition}
\begin{proof}
A maximal legal position in the game $(R,B)+(R',B')$ is one where both the pieces placed in $(R,B)$ and the ones placed in $(R',B')$ form maximal legal positions. Thus a facet in the legal complex of $(R,B)+(R',B')$ is a union of a facet of $\Delta_{R,B}$ and a facet of $\Delta_{R',B'}$.
\end{proof}

\vspace{1em}

A natural and important question is whether any given simplicial complex $\Delta$ is the legal or illegal complex of some game. We will answer this question positively in both cases. This will allow us to view properties of games as properties of simplicial complexes and vice-versa. We are able to show this for any bipartition of the vertices into Left $\L$ and Right $\R$, where $\L$ or $\R$ could even be the empty set.

\begin{proposition}[\textbf{Games from Simplicial Complexes}]\label{thm:legalillegal}
Given an $(\L,\R)$-labelled simplicial complex $\Delta$, there exist two SP-games $(R_1,B)$ and $(R_2,B)$ such that
\begin{itemize}
\item[(a)] $\Delta=\Delta_{R_1,B}$ and
\item[(b)] $\Delta=\Gamma_{R_2,B}$
\end{itemize}
and the sets of Left (respectively Right) positions is $\L$ (respectively $\R$).
\end{proposition}
\begin{proof}
Let $m=|\L|$ and $n=|\R|$. Let $B$ be the board consisting of $m$ disjoint $3$-cycles and $n$ disjoint $4$-cycles. In the games $(R_1,B)$ and $(R_2,B)$, Left will be playing $3$-cycles, while Right will be playing $4$-cycles.

In $\Delta$, label the vertices belonging to $\L$ as $1, \ldots, m$, and the vertices in $\R$ as $m+1, \ldots, n+m$. Similarly, label the $3$-cycles of $B$ as $1, \ldots, m$, and the $4$-cycles as $m+1, \ldots, n+m$. 

(a) In $R_1$, playing on a set of cycles of $B$ is legal if and only if the corresponding set of vertices in $\Delta$ forms a face.

(b) In $R_2$, playing on a set of cycles of $B$ is legal if and only if the corresponding set of vertices in $\Delta$ does not contain a facet.

It is now easy to see that $\Delta=\Delta_{R_1,B}$ and $\Delta=\Gamma_{R_2,B}$.
\end{proof}

As seen above, it is rather simple to construct games on fixed boards from simplicial complexes by restricting the legal moves to certain parts of the board. We now move on to look at games where such restrictions can be relaxed. We call these invariant games.

\section{Invariant Games}\label{sec:iSP}

As we have shown in the previous section, every $(\L,\R)$-labelled simplicial complex is the legal or illegal complex of some SP-game and board. The rules created as part of this construction, however, depend heavily on the board. We now define the concept of invariance for SP-games, which in a sense forces the ruleset to be ``uniform'' across the board.

\begin{definition}
The ruleset of an SP-game is \textbf{invariant} if the following conditions hold:
\begin{itemize}
\item Every basic position is legal.
\item The ruleset does not depend on the board, \ie if $B_1$ and $B_2$ are isomorphic boards, then a move in $B_1$ is legal if and only if its isomorphic image in $B_2$ is legal.
\end{itemize}
\end{definition}

Note that in particular the second condition has to hold for \textit{any} board isomorphism between $B_1$ and $B_2$. If the ruleset of an SP-game is invariant, we also say that the game is an \textbf{invariant strong placement game (iSP-game)}.

\col and \snort are examples of rulesets that are invariant, while \nogo is not. 
That \nogo is not invariant is not immediately obvious. Indeed on most boards both conditions hold, but whenever the board has an isolated vertex, playing on it is illegal (thus the basic position corresponding to that vertex is illegal).

An example of an SP-game which fails the second condition is the following.
\begin{example}
Consider playing on the boards $B_1$ and $B_2$, both 4-cycles with labels as below, a game in which a Left piece on vertex 1 cannot be adjacent to another piece. 
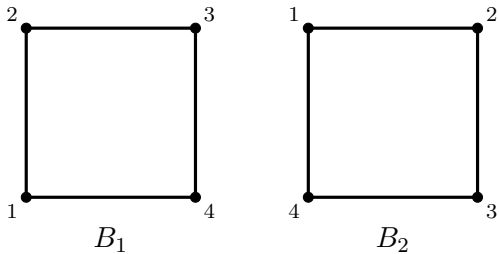
\begin{center}
\begin{tikzpicture}[line cap=round,line join=round,>=triangle 45,x=1.0cm,y=1.0cm,scale=0.75]
\draw [line width=1.2pt] (0.,0.)-- (0.,3.)-- (3.,3.)-- (3.,0.)--cycle;
\draw [line width=1.2pt] (5.,3.)-- (5.,0.)-- (8.,0.)-- (8.,3.)--cycle;
\draw (1.5,-0.75) node {$B_1$};
\draw (6.5,-0.75) node {$B_2$};
\begin{scriptsize}
\draw [fill=black] (0.,0.) circle (2.5pt);
\draw[color=black] (-0.25,-0.25) node {$1$};
\draw [fill=black] (0.,3.) circle (2.5pt);
\draw[color=black] (-0.25,3.25) node {$2$};
\draw [fill=black] (3.,3.) circle (2.5pt);
\draw[color=black] (3.25,3.25) node {$3$};
\draw [fill=black] (3.,0.) circle (2.5pt);
\draw[color=black] (3.25,-0.25) node {$4$};
\draw [fill=black] (5.,3.) circle (2.5pt);
\draw[color=black] (4.75,3.25) node {$1$};
\draw [fill=black] (8.,3.) circle (2.5pt);
\draw[color=black] (8.25,3.25) node {$2$};
\draw [fill=black] (8.,0.) circle (2.5pt);
\draw[color=black] (8.25,-0.25) node {$3$};
\draw [fill=black] (5.,0.) circle (2.5pt);
\draw[color=black] (4.75,-0.25) node {$4$};
\end{scriptsize}
\end{tikzpicture}
\end{center}

Now $B_1$ and $B_2$ are isomorphic graphs and, since neither contains pieces, also isomorphic boards. The position in which there is a Left piece in the top left corner and a Right piece in the top right corner is legal on $B_1$ but not on $B_2$. Thus this game is not invariant.
\end{example}

Similar to the question of the previous section, we are interested in which simplicial complexes appear as the legal or illegal complex of an iSP-game.

We will show below that the illegal complex of an iSP-game cannot contain an isolated vertex.

\begin{proposition}\label{thm:illegalinvariantnovertices}
Let $\Gamma$ be a simplicial complex. If $\Gamma$ is the illegal complex of some iSP-game then $\Gamma$ has no facets that are one-element sets.
\end{proposition}
\begin{proof}
Assume that $\Gamma$ has a facet that is a one-element set, \ie an isolated vertex, and label this vertex $a$. If $\Gamma$ is the illegal complex of some SP-game $(R,B)$, then since $\{a\}$ is a facet of $\Gamma$, there exists a basic position (corresponding to the vertex $a$) which is illegal. Thus $G$ does not satisfy the first condition of invariance.
\end{proof}

Other than the isolated vertex situation, there is no obstruction for a simplicial complex $\Gamma$ being an illegal complex. We set out to prove this by constructing a $\Gamma$-board and a $\Gamma$-ruleset.

\begin{construction}[\textbf{$\mathbf{\Gamma}$-board}]
Given an $(\L,\R)$-labelled simplicial complex $\Gamma$ with no isolated vertices we can construct a graph $B_\Gamma$ (called the \textbf{$\Gamma$-board}) as follows:

If $\Gamma$ is empty, then let $B_\Gamma$ be empty.

If $\Gamma$ is non-empty, then let $H=\Gamma^{[1]}$, \ie the underlying graph of $\Gamma$. Let $n$ be the number of vertices in the graph $H$ and (re)label the vertices of $H$ as $1,\ldots, n$. Begin constructing the board $B_\Gamma$ by using $n$ cycles of sizes $n^4+4$ and $n^4+5$ and label these $1,\ldots, n$ so that cycle $i$ will have size $n^4+4$ if the vertex $i$ in $H$ belongs to $\L$, and size $n^4+5$ if the vertex $i$ belongs to $\R$. For each cycle, designate $n-1$ consecutive vertices for joining, called \textbf{connection vertices} (see \cref{fig:cyclei}).

\begin{figure}[!ht]
\begin{center}
\definecolor{qqqqff}{rgb}{0.,0.,1.}
\definecolor{qqzzqq}{rgb}{0.,0.6,0.}
\definecolor{ffqqqq}{rgb}{1.,0.,0.}
\begin{tikzpicture}[line cap=round,line join=round,>=triangle 45,x=1.0cm,y=1.0cm,scale=0.75]
\clip(-4,-7) rectangle (14,6);
\draw [line width=0.5pt] (5.,0.) circle (5.cm);
\draw [line width=1.2pt,dash pattern=on 2pt off 2pt,color=qqqqff] (0.84,-1.72) circle (0.4984540207992376cm);
\draw [line width=1.2pt,dash pattern=on 2pt off 2pt,color=qqqqff] (3.28,-4.16) circle (0.4984540207992371cm);
\draw [line width=1.2pt,dash pattern=on 2pt off 2pt,color=qqqqff] (5.,-4.5) circle (0.5cm);
\draw [line width=1.2pt,dash pattern=on 2pt off 2pt,color=qqqqff] (6.72,-4.16) circle (0.4984540207992374cm);
\draw [line width=1.2pt,dash pattern=on 2pt off 2pt,color=qqqqff] (8.19,-3.17) circle (0.5029997182261414cm);
\draw [line width=1.2pt,dash pattern=on 2pt off 2pt,color=qqqqff] (9.5,0.) circle (0.5cm);
\draw [color=qqzzqq](10,5.4622222222222225) node {$n^4-n+5$ or $n^4-n+6$};
\draw [color=qqzzqq](10,4.8) node {outer vertices};
\draw [color=qqqqff](5,-2) node {inner cycles ($n^3$ vertices)};
\draw [color=ffqqqq](6.06,-6.1377777777777816) node {$n-1$ connection vertices};
\draw (5,0) node {\Large $i$};
\begin{scriptsize}
\draw [fill=ffqqqq] (10.,0.) circle (2.5pt);
\draw[color=ffqqqq] (10.48,0) node {$i,n$};
\draw [fill=qqzzqq] (5.,5.) ++(-2.5pt,0 pt) -- ++(2.5pt,2.5pt)--++(2.5pt,-2.5pt)--++(-2.5pt,-2.5pt)--++(-2.5pt,2.5pt);
\draw [fill=qqzzqq] (0.,0.) ++(-2.5pt,0 pt) -- ++(2.5pt,2.5pt)--++(2.5pt,-2.5pt)--++(-2.5pt,-2.5pt)--++(-2.5pt,2.5pt);
\draw [fill=ffqqqq] (5.,-5.) circle (2.5pt);
\draw[color=ffqqqq] (5.16,-5.407777777777781) node {$i,i-1$};
\draw [fill=qqzzqq] (8.535533905932738,3.5355339059327373) ++(-2.5pt,0 pt) -- ++(2.5pt,2.5pt)--++(2.5pt,-2.5pt)--++(-2.5pt,-2.5pt)--++(-2.5pt,2.5pt);
\draw [fill=qqzzqq] (9.619397662556434,1.913417161825449) ++(-2.5pt,0 pt) -- ++(2.5pt,2.5pt)--++(2.5pt,-2.5pt)--++(-2.5pt,-2.5pt)--++(-2.5pt,2.5pt);
\draw [fill=qqzzqq] (6.913417161825449,4.619397662556434) ++(-2.5pt,0 pt) -- ++(2.5pt,2.5pt)--++(2.5pt,-2.5pt)--++(-2.5pt,-2.5pt)--++(-2.5pt,2.5pt);
\draw [fill=qqzzqq] (3.0865828381745515,4.619397662556434) ++(-2.5pt,0 pt) -- ++(2.5pt,2.5pt)--++(2.5pt,-2.5pt)--++(-2.5pt,-2.5pt)--++(-2.5pt,2.5pt);
\draw [fill=qqzzqq] (1.4644660940672627,3.5355339059327378) ++(-2.5pt,0 pt) -- ++(2.5pt,2.5pt)--++(2.5pt,-2.5pt)--++(-2.5pt,-2.5pt)--++(-2.5pt,2.5pt);
\draw [fill=qqzzqq] (0.3806023374435661,1.9134171618254494) ++(-2.5pt,0 pt) -- ++(2.5pt,2.5pt)--++(2.5pt,-2.5pt)--++(-2.5pt,-2.5pt)--++(-2.5pt,2.5pt);
\draw [fill=ffqqqq] (0.3806023374435661,-1.9134171618254483) circle (2.5pt);
\draw[color=ffqqqq] (-0.12,-2) node {$i,1$};
\draw [fill=ffqqqq] (1.4644660940672614,-3.5355339059327373) circle (2.5pt);
\draw [fill=ffqqqq] (3.0865828381745524,-4.619397662556434) circle (2.5pt);
\draw[color=ffqqqq] (2.8,-5) node {$i,i-2$};
\draw [fill=ffqqqq] (6.91341716182545,-4.619397662556433) circle (2.5pt);
\draw[color=ffqqqq] (7.3,-4.987777777777781) node {$i,i+1$};
\draw [fill=ffqqqq] (9.619397662556434,-1.913417161825449) circle (2.5pt);
\draw [fill=ffqqqq] (8.535533905932738,-3.5355339059327373) circle (2.5pt);
\draw[color=ffqqqq] (9.4,-3.667777777777781) node {$i,i+2$};
\draw [fill=black] (9.158315414732623,-1.7200037533671415) circle (1.5pt);
\draw [fill=black] (9.292399973803288,-1.3510375512521906) circle (1.5pt);
\draw [fill=black] (8.992583564496345,-2.0758796883523063) circle (1.5pt);
\draw [fill=black] (1.819605320654964,-3.183565561378278) circle (1.5pt);
\draw [fill=black] (1.5542416615094896,-2.894261472763495) circle (1.5pt);
\draw [fill=black] (2.1091737036850677,-3.448640793781528) circle (1.5pt);
\end{scriptsize}
\end{tikzpicture}
\end{center}
\caption{Cycle $i$ in the Board $B_\Gamma$}
\label{fig:cyclei}
\end{figure}
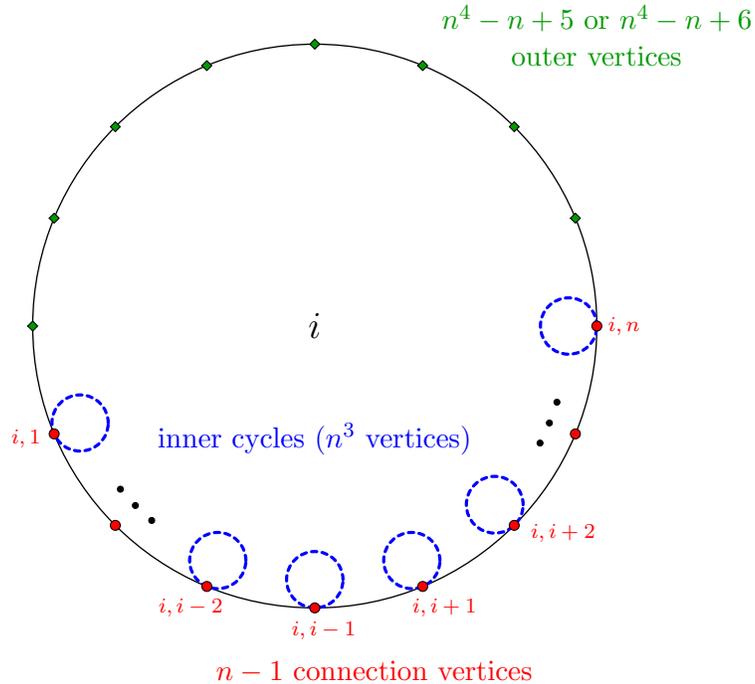

Call the remaining vertices \textbf{outer vertices}. To each connection vertex, join a cycle of length $n^3$ (called \textbf{inner cycles}). In cycle $i$ label the connection vertices as $i,j$ where $j=1,\ldots, n$ and $j\neq i$.

Label the edges in $H$ as $1,\ldots, k$. If the endpoints of the edge $l$ are the vertices $i$ and $j$, then add a path of $2+l$ vertices to $B_\Gamma$, whose end vertices are $i,j$ and $j,i$ (see \cref{fig:edgetopath}). The $l$ vertices between $i,j$ and $j,i$ are called \textbf{centre vertices}.

\begin{figure}[!ht]
\begin{center}
\definecolor{qqqqqq}{rgb}{0.,0.,0.}
\definecolor{qqqqff}{rgb}{0.,0.,1.}
\definecolor{ffqqqq}{rgb}{1.,0.,0.}
\begin{tikzpicture}[line cap=round,line join=round,>=triangle 45,x=1.0cm,y=1.0cm,scale=0.5]
\draw [line width=1pt] (5.,0.) circle (3.cm);
\draw [line width=1pt,dash pattern=on 2pt off 2pt] (5.,-2.5) circle (0.5cm);
\draw [line width=1pt,dash pattern=on 2pt off 2pt] (6.82,-1.7477777777777739) circle (0.4799249919454108cm);
\draw [line width=1pt,dash pattern=on 2pt off 2pt] (7.5,0.) circle (0.5cm);
\draw (5,0) node {$i$};
\draw [line width=1pt] (13.,-8.) circle (3.cm);
\draw (13,-8) node {$j$};
\draw [line width=1pt,dash pattern=on 2pt off 2pt] (13.,-5.5) circle (0.5cm);
\draw [line width=1pt,dash pattern=on 2pt off 2pt] (11.26,-6.2) circle (0.4986501454905335cm);
\draw [line width=1pt,dash pattern=on 2pt off 2pt] (10.5,-8.) circle (0.5cm);
\draw [line width=1pt,color=qqqqff] (7.121320343559643,-2.1213203435596424)-- (10.878679656440358,-5.878679656440357);
\draw [line width=0.8pt,color=qqqqqq] (5.02,-3.0477777777777724)-- (4.94,-3.0877777777777724)-- (4.88,-3.1077777777777724)-- (4.82,-3.1477777777777725)-- (4.76,-3.227777777777772)-- (4.7,-3.247777777777772)-- (4.64,-3.267777777777772)-- (4.56,-3.267777777777772)-- (4.52,-3.327777777777772)-- (4.48,-3.407777777777772)-- (4.44,-3.467777777777772)-- (4.44,-3.547777777777772)-- (4.42,-3.6077777777777715)-- (4.42,-3.6877777777777716)-- (4.42,-3.7877777777777717)-- (4.42,-3.8677777777777713)-- (4.42,-3.9477777777777714)-- (4.42,-4.0277777777777715)-- (4.44,-4.127777777777771)-- (4.52,-4.127777777777771)-- (4.6,-4.127777777777771)-- (4.68,-4.127777777777771)-- (4.76,-4.127777777777771)-- (4.88,-4.127777777777771)-- (4.96,-4.127777777777771)-- (5.02,-4.147777777777771)-- (5.04,-4.207777777777771)-- (5.08,-4.267777777777771)-- (5.12,-4.327777777777771)-- (5.18,-4.36777777777777)-- (5.22,-4.427777777777771)-- (5.24,-4.4877777777777705)-- (5.26,-4.54777777777777)-- (5.28,-4.607777777777771)-- (5.36,-4.607777777777771)-- (5.44,-4.607777777777771)-- (5.54,-4.607777777777771)-- (5.62,-4.607777777777771)-- (5.7,-4.58777777777777)-- (5.76,-4.567777777777771)-- (5.86,-4.567777777777771)-- (5.94,-4.54777777777777)-- (6.02,-4.527777777777771)-- (6.08,-4.4877777777777705)-- (6.12,-4.427777777777771)-- (6.18,-4.387777777777771)-- (6.26,-4.387777777777771)-- (6.34,-4.387777777777771)-- (6.42,-4.387777777777771)-- (6.5,-4.387777777777771)-- (6.58,-4.387777777777771)-- (6.66,-4.387777777777771)-- (6.74,-4.387777777777771)-- (6.82,-4.387777777777771)-- (6.9,-4.387777777777771)-- (6.98,-4.327777777777771)-- (7.02,-4.267777777777771)-- (7.08,-4.127777777777771)-- (7.12,-3.9877777777777714)-- (7.14,-3.8677777777777713)-- (7.14,-3.7477777777777717)-- (7.14,-3.6677777777777716)-- (7.1,-3.6077777777777715)-- (7.02,-3.6077777777777715)-- (6.94,-3.5877777777777715)-- (6.86,-3.5877777777777715)-- (6.78,-3.5877777777777715)-- (6.7,-3.5877777777777715)-- (6.62,-3.5877777777777715)-- (6.54,-3.5877777777777715)-- (6.46,-3.5877777777777715)-- (6.38,-3.5877777777777715)-- (6.3,-3.6077777777777715)-- (6.22,-3.6077777777777715)-- (6.18,-3.547777777777772)-- (6.18,-3.467777777777772)-- (6.18,-3.387777777777772)-- (6.22,-3.287777777777772)-- (6.24,-3.227777777777772)-- (6.28,-3.167777777777772)-- (6.32,-3.1077777777777724)-- (6.34,-2.9877777777777723)-- (6.34,-2.9077777777777727)-- (6.34,-2.8277777777777726)-- (6.28,-2.8077777777777726)-- (6.2,-2.8077777777777726)-- (6.12,-2.8077777777777726);
\draw [line width=0.8pt,color=qqqqqq] (7.82,-1.1477777777777747)-- (7.84,-1.2077777777777745)-- (7.92,-1.3077777777777744)-- (7.96,-1.3677777777777744)-- (8.04,-1.4277777777777743)-- (8.12,-1.4677777777777743)-- (8.28,-1.507777777777774)-- (8.46,-1.507777777777774)-- (8.58,-1.507777777777774)-- (8.68,-1.507777777777774)-- (8.76,-1.4877777777777743)-- (8.84,-1.4677777777777743)-- (8.9,-1.4077777777777742)-- (8.96,-1.3677777777777744)-- (9.02,-1.3277777777777744)-- (9.06,-1.2677777777777746)-- (9.12,-1.2277777777777745)-- (9.18,-1.2077777777777745)-- (9.22,-1.1477777777777747)-- (9.3,-1.1477777777777747)-- (9.38,-1.1477777777777747)-- (9.42,-1.2077777777777745)-- (9.46,-1.2677777777777746)-- (9.54,-1.2877777777777746)-- (9.62,-1.2877777777777746)-- (9.7,-1.2677777777777746)-- (9.74,-1.2077777777777745)-- (9.82,-1.1277777777777747)-- (9.82,-1.0477777777777748)-- (9.84,-0.9877777777777749)-- (9.84,-0.9077777777777749)-- (9.84,-0.8277777777777751)-- (9.84,-0.7477777777777751)-- (9.84,-0.6677777777777752)-- (9.78,-0.6077777777777753)-- (9.7,-0.5477777777777754)-- (9.66,-0.4877777777777754)-- (9.58,-0.46777777777777546)-- (9.5,-0.46777777777777546)-- (9.42,-0.46777777777777546)-- (9.32,-0.46777777777777546)-- (9.26,-0.4477777777777755)-- (9.2,-0.4277777777777755)-- (9.14,-0.4077777777777755)-- (9.1,-0.26777777777777567)-- (9.1,-0.1877777777777758)-- (9.1,-0.0877777777777759)-- (9.1,0.012222222222223977)-- (9.1,0.09222222222222388)-- (9.1,0.17222222222222378)-- (9.1,0.27222222222222364)-- (9.1,0.35222222222222355)-- (9.04,0.37222222222222356)-- (8.96,0.37222222222222356)-- (8.88,0.37222222222222356)-- (8.8,0.37222222222222356)-- (8.72,0.37222222222222356)-- (8.64,0.35222222222222355)-- (8.58,0.3122222222222236)-- (8.52,0.27222222222222364)-- (8.46,0.2522222222222237)-- (8.4,0.21222222222222373)-- (8.34,0.19222222222222377)-- (8.28,0.13222222222222382)-- (8.2,0.13222222222222382)-- (8.14,0.09222222222222388)-- (8.1,0.03222222222222395)-- (8.1,-0.047777777777775955)-- (8.06,-0.10777777777777588);
\draw [line width=0.8pt,color=qqqqqq] (10.22,-6.827777777777768)-- (10.2,-6.727777777777768)-- (10.14,-6.587777777777768)-- (10.1,-6.487777777777768)-- (10.04,-6.407777777777769)-- (10.,-6.347777777777768)-- (9.94,-6.287777777777769)-- (9.9,-6.227777777777769)-- (9.84,-6.187777777777769)-- (9.78,-6.1677777777777685)-- (9.72,-6.147777777777769)-- (9.66,-6.127777777777768)-- (9.58,-6.127777777777768)-- (9.5,-6.127777777777768)-- (9.42,-6.127777777777768)-- (9.34,-6.127777777777768)-- (9.28,-6.147777777777769)-- (9.22,-6.1677777777777685)-- (9.14,-6.1677777777777685)-- (9.06,-6.1677777777777685)-- (9.,-6.147777777777769)-- (8.92,-6.107777777777769)-- (8.86,-6.027777777777769)-- (8.76,-5.967777777777769)-- (8.66,-5.927777777777769)-- (8.58,-5.847777777777769)-- (8.5,-5.807777777777769)-- (8.42,-5.767777777777769)-- (8.3,-5.747777777777769)-- (8.2,-5.747777777777769)-- (8.12,-5.747777777777769)-- (8.,-5.747777777777769)-- (7.94,-5.807777777777769)-- (7.8,-5.947777777777769)-- (7.74,-6.027777777777769)-- (7.72,-6.107777777777769)-- (7.62,-6.347777777777768)-- (7.6,-6.567777777777768)-- (7.6,-6.707777777777768)-- (7.6,-6.847777777777768)-- (7.6,-6.927777777777767)-- (7.6,-7.0077777777777674)-- (7.66,-7.0477777777777675)-- (7.76,-7.1277777777777676)-- (7.94,-7.207777777777768)-- (8.2,-7.267777777777767)-- (8.5,-7.367777777777767)-- (8.6,-7.387777777777767)-- (8.62,-7.447777777777767)-- (8.62,-7.547777777777767)-- (8.62,-7.787777777777767)-- (8.62,-7.8877777777777665)-- (8.62,-8.027777777777766)-- (8.6,-8.187777777777766)-- (8.6,-8.427777777777766)-- (8.58,-8.687777777777766)-- (8.54,-8.887777777777766)-- (8.48,-9.007777777777765)-- (8.42,-9.107777777777764)-- (8.42,-9.187777777777764)-- (8.42,-9.267777777777765)-- (8.46,-9.327777777777765)-- (8.56,-9.327777777777765)-- (8.74,-9.327777777777765)-- (8.86,-9.327777777777765)-- (8.96,-9.327777777777765)-- (9.04,-9.307777777777765)-- (9.12,-9.267777777777765)-- (9.2,-9.207777777777766)-- (9.22,-9.147777777777765)-- (9.24,-9.087777777777765)-- (9.32,-8.987777777777765)-- (9.38,-8.907777777777765)-- (9.38,-8.807777777777765)-- (9.44,-8.747777777777765)-- (9.48,-8.667777777777765)-- (9.5,-8.567777777777765)-- (9.5,-8.407777777777765)-- (9.52,-8.327777777777767)-- (9.56,-8.267777777777766)-- (9.62,-8.207777777777766)-- (9.7,-8.167777777777767)-- (9.76,-8.127777777777766)-- (9.84,-8.127777777777766)-- (9.88,-8.067777777777767)-- (9.9,-8.007777777777767)-- (9.94,-7.947777777777767);
\draw [color=qqqqff](12,-2.8) node {centre vertices};
\draw [line width=0.8pt,color=qqqqqq] (11.82,-5.157777777777765)-- (11.76,-5.137777777777766)-- (11.7,-5.0577777777777655)-- (11.66,-4.957777777777766)-- (11.6,-4.837777777777766)-- (11.6,-4.6377777777777665)-- (11.6,-4.417777777777767)-- (11.66,-4.277777777777767)-- (11.68,-4.177777777777767)-- (11.7,-4.117777777777768)-- (11.78,-4.057777777777767)-- (11.9,-3.9977777777777677)-- (12.02,-3.957777777777768)-- (12.1,-3.957777777777768)-- (12.16,-3.897777777777768)-- (12.22,-3.857777777777768)-- (12.4,-3.7777777777777684)-- (12.52,-3.6977777777777683)-- (12.62,-3.6777777777777683)-- (12.72,-3.6377777777777687)-- (12.86,-3.5977777777777686)-- (12.96,-3.5977777777777686)-- (13.04,-3.5577777777777686)-- (13.08,-3.497777777777769)-- (13.14,-3.437777777777769)-- (13.16,-3.3377777777777693)-- (13.24,-3.2777777777777692)-- (13.36,-3.237777777777769)-- (13.5,-3.237777777777769)-- (13.62,-3.237777777777769)-- (13.74,-3.2577777777777692)-- (13.82,-3.2977777777777693)-- (13.88,-3.397777777777769)-- (13.94,-3.437777777777769)-- (14.02,-3.5377777777777686)-- (14.1,-3.5377777777777686)-- (14.2,-3.5177777777777686)-- (14.3,-3.477777777777769)-- (14.38,-3.477777777777769)-- (14.38,-3.5777777777777686)-- (14.38,-3.6577777777777687)-- (14.3,-3.7977777777777684)-- (14.24,-3.937777777777768)-- (14.14,-4.097777777777767)-- (14.12,-4.177777777777767)-- (14.1,-4.237777777777767)-- (14.06,-4.2977777777777675)-- (14.,-4.357777777777767)-- (13.92,-4.437777777777767)-- (13.86,-4.457777777777767)-- (13.78,-4.557777777777766)-- (13.72,-4.597777777777766)-- (13.64,-4.697777777777766)-- (13.56,-4.757777777777767)-- (13.5,-4.777777777777766)-- (13.44,-4.797777777777766)-- (13.36,-4.797777777777766)-- (13.28,-4.797777777777766)-- (13.2,-4.797777777777766)-- (13.12,-4.817777777777766)-- (13.04,-4.857777777777766)-- (12.98,-4.897777777777766)-- (12.92,-4.917777777777766);
\draw [line width=1pt,color=qqqqff] (-10.,-5.)-- (-4.,-5.);
\draw [line width=0.8pt,color=qqqqqq] (-4.,-4.897777777777766)-- (-4.,-4.817777777777766)-- (-4.,-4.337777777777767)-- (-3.98,-3.9977777777777677)-- (-3.96,-3.897777777777768)-- (-3.9,-3.6977777777777683)-- (-3.76,-3.5577777777777686)-- (-3.68,-3.477777777777769)-- (-3.56,-3.397777777777769)-- (-3.42,-3.3377777777777693)-- (-3.32,-3.3377777777777693)-- (-3.1,-3.3177777777777693)-- (-2.92,-3.3177777777777693)-- (-2.78,-3.3177777777777693)-- (-2.66,-3.3177777777777693)-- (-2.6,-3.3377777777777693)-- (-2.52,-3.3377777777777693)-- (-2.44,-3.3377777777777693)-- (-2.34,-3.3377777777777693)-- (-2.14,-3.3377777777777693)-- (-2.,-3.2577777777777692)-- (-1.88,-3.1977777777777696)-- (-1.8,-3.1977777777777696)-- (-1.72,-3.1777777777777696)-- (-1.6,-3.1777777777777696)-- (-1.5,-3.217777777777769)-- (-1.42,-3.2777777777777692)-- (-1.3,-3.377777777777769)-- (-1.18,-3.6977777777777683)-- (-1.02,-4.137777777777767)-- (-0.94,-4.397777777777767)-- (-0.92,-4.497777777777767)-- (-0.84,-4.617777777777767)-- (-0.74,-4.897777777777766)-- (-0.7,-5.037777777777766)-- (-0.68,-5.137777777777766)-- (-0.68,-5.357777777777765)-- (-0.68,-5.4777777777777645)-- (-0.7,-5.537777777777765)-- (-0.76,-5.617777777777764)-- (-0.86,-5.737777777777764)-- (-0.9,-5.797777777777764)-- (-0.94,-5.917777777777764)-- (-1.02,-6.157777777777763)-- (-1.14,-6.397777777777763)-- (-1.2,-6.517777777777763)-- (-1.26,-6.597777777777762)-- (-1.34,-6.6577777777777625)-- (-1.4,-6.6977777777777625)-- (-1.48,-6.717777777777762)-- (-1.68,-6.817777777777762)-- (-1.74,-6.837777777777762)-- (-1.88,-6.917777777777761)-- (-1.98,-7.057777777777761)-- (-2.02,-7.137777777777761)-- (-2.1,-7.337777777777761)-- (-2.18,-7.55777777777776)-- (-2.2,-7.75777777777776)-- (-2.24,-7.91777777777776)-- (-2.3,-7.977777777777759)-- (-2.36,-8.037777777777759)-- (-2.48,-8.037777777777759)-- (-2.78,-7.977777777777759)-- (-3.,-7.89777777777776)-- (-3.16,-7.85777777777776)-- (-3.3,-7.75777777777776)-- (-3.42,-7.67777777777776)-- (-3.5,-7.59777777777776)-- (-3.64,-7.357777777777761)-- (-3.8,-7.157777777777762)-- (-3.9,-6.997777777777761)-- (-4.04,-6.857777777777762)-- (-4.14,-6.757777777777762)-- (-4.2,-6.6577777777777625)-- (-4.24,-6.557777777777763)-- (-4.26,-6.497777777777762)-- (-4.3,-6.417777777777763)-- (-4.32,-6.337777777777763)-- (-4.36,-6.2777777777777635)-- (-4.36,-6.197777777777763)-- (-4.38,-6.097777777777764)-- (-4.38,-5.957777777777764)-- (-4.38,-5.877777777777764)-- (-4.38,-5.797777777777764)-- (-4.38,-5.717777777777764)-- (-4.32,-5.677777777777765)-- (-4.28,-5.617777777777764)-- (-4.2,-5.537777777777765)-- (-4.14,-5.4777777777777645)-- (-4.08,-5.417777777777765)-- (-4.,-5.377777777777765)-- (-3.96,-5.217777777777766)-- (-3.94,-5.0977777777777655)-- (-3.94,-4.997777777777766);
\draw [line width=0.8pt,color=qqqqqq] (-10.08,-4.977777777777765)-- (-10.08,-4.877777777777766)-- (-10.14,-4.657777777777766)-- (-10.22,-4.457777777777767)-- (-10.28,-4.337777777777767)-- (-10.32,-4.237777777777767)-- (-10.36,-4.157777777777768)-- (-10.42,-4.077777777777768)-- (-10.48,-4.017777777777768)-- (-10.54,-3.9777777777777676)-- (-10.6,-3.897777777777768)-- (-10.66,-3.857777777777768)-- (-10.72,-3.837777777777768)-- (-10.78,-3.7977777777777684)-- (-10.88,-3.7977777777777684)-- (-10.96,-3.7977777777777684)-- (-11.06,-3.7977777777777684)-- (-11.26,-3.857777777777768)-- (-11.32,-3.937777777777768)-- (-11.38,-4.057777777777767)-- (-11.48,-4.217777777777767)-- (-11.54,-4.2577777777777674)-- (-11.68,-4.237777777777767)-- (-11.74,-4.197777777777767)-- (-12.,-4.157777777777768)-- (-12.08,-4.137777777777767)-- (-12.22,-4.137777777777767)-- (-12.3,-4.137777777777767)-- (-12.38,-4.057777777777767)-- (-12.46,-3.9977777777777677)-- (-12.52,-3.9777777777777676)-- (-12.6,-3.9777777777777676)-- (-12.68,-3.9777777777777676)-- (-12.76,-3.9777777777777676)-- (-12.84,-3.9777777777777676)-- (-12.92,-3.9977777777777677)-- (-13.,-4.057777777777767)-- (-13.04,-4.177777777777767)-- (-13.04,-4.2977777777777675)-- (-13.04,-4.577777777777767)-- (-13.02,-4.6377777777777665)-- (-13.,-4.837777777777766)-- (-12.96,-4.897777777777766)-- (-12.9,-5.0177777777777655)-- (-12.8,-5.197777777777765)-- (-12.74,-5.317777777777765)-- (-12.66,-5.4777777777777645)-- (-12.62,-5.557777777777765)-- (-12.6,-5.637777777777765)-- (-12.54,-5.717777777777764)-- (-12.54,-5.797777777777764)-- (-12.52,-5.8577777777777635)-- (-12.52,-5.937777777777764)-- (-12.46,-6.017777777777764)-- (-12.4,-6.077777777777763)-- (-12.32,-6.137777777777764)-- (-12.26,-6.217777777777763)-- (-12.22,-6.437777777777763)-- (-12.22,-6.597777777777762)-- (-12.22,-6.777777777777762)-- (-12.22,-6.857777777777762)-- (-12.24,-6.917777777777761)-- (-12.24,-6.997777777777761)-- (-12.26,-7.057777777777761)-- (-12.26,-7.137777777777761)-- (-12.26,-7.217777777777761)-- (-12.24,-7.277777777777761)-- (-12.12,-7.3777777777777604)-- (-12.06,-7.437777777777761)-- (-11.88,-7.51777777777776)-- (-11.68,-7.51777777777776)-- (-11.5,-7.51777777777776)-- (-11.36,-7.51777777777776)-- (-11.12,-7.397777777777761)-- (-10.98,-7.337777777777761)-- (-10.92,-7.317777777777761)-- (-10.78,-7.257777777777761)-- (-10.54,-7.137777777777761)-- (-10.42,-7.0777777777777615)-- (-10.36,-7.0377777777777615)-- (-10.3,-6.957777777777761)-- (-10.22,-6.837777777777762)-- (-10.16,-6.677777777777762)-- (-10.12,-6.617777777777762)-- (-10.06,-6.437777777777763)-- (-10.02,-6.297777777777763)-- (-9.98,-6.237777777777763)-- (-9.92,-6.137777777777764)-- (-9.84,-6.117777777777763)-- (-9.82,-6.037777777777763)-- (-9.82,-5.937777777777764)-- (-9.82,-5.757777777777764)-- (-9.82,-5.657777777777764)-- (-9.82,-5.497777777777765)-- (-9.88,-5.337777777777765)-- (-9.92,-5.277777777777765)-- (-9.98,-5.257777777777765)-- (-10.,-5.197777777777765)-- (-10.06,-5.137777777777766)-- (-10.1,-5.077777777777766)-- (-10.12,-5.0177777777777655);
\draw (0.5,-4) node {\Large $\Rightarrow$};
\draw (-7,-10) node {\Large $H$};
\draw (5.84,-10) node {\Large $B_{\Gamma}$};
\draw[color=ffqqqq] (8.25,-2.1) node {$i,j$};
\draw[color=ffqqqq] (10.8,-4.6) node {$j,i$};
\draw[color=ffqqqq] (-9.5,-4.5) node {$i$};
\draw[color=ffqqqq] (-4.5,-4.5) node {$j$};
\draw [fill=black] (8.,0.) circle (4pt);
\draw [fill=black] (5.,3.) circle (4pt);
\draw [fill=black] (2.,0.) circle (4pt);
\draw [fill=black] (5.,-3.) circle (4pt);
\draw [fill=black] (7.121320343559643,2.1213203435596424) circle (4pt);
\draw [fill=black] (7.77163859753386,1.1480502970952693) circle (4pt);
\draw [fill=black] (6.14805029709527,2.77163859753386) circle (4pt);
\draw [fill=black] (3.8519497029047307,2.77163859753386) circle (4pt);
\draw [fill=black] (2.8786796564403576,2.121320343559643) circle (4pt);
\draw [fill=black] (2.22836140246614,1.1480502970952697) circle (4pt);
\draw [fill=black] (2.2283614024661396,-1.148050297095269) circle (4pt);
\draw [fill=black] (2.878679656440357,-2.1213203435596424) circle (4pt);
\draw [fill=black] (3.8519497029047316,-2.7716385975338604) circle (4pt);
\draw [fill=black] (6.14805029709527,-2.77163859753386) circle (4pt);
\draw [fill=black] (7.77163859753386,-1.1480502970952693) circle (4pt);
\draw [fill=ffqqqq] (7.121320343559643,-2.1213203435596424) circle (5pt);
\draw [fill=black] (7.310175230407013,-0.9555576407595231) circle (1.5pt);
\draw [fill=black] (7.384666652112939,-0.7505764173623282) circle (1.5pt);
\draw [fill=black] (7.218101980275747,-1.153266493529059) circle (1.5pt);
\draw [fill=black] (5.957859283554317,-2.309221858745671) circle (1.5pt);
\draw [fill=black] (5.7529523935266615,-2.383917509705919) circle (1.5pt);
\draw [fill=black] (6.155476286062875,-2.216951635094087) circle (1.5pt);
\draw [fill=black] (16.,-8.) circle (4pt);
\draw [fill=black] (13.,-5.) circle (4pt);
\draw [fill=black] (10.,-8.) circle (4pt);
\draw [fill=black] (13.,-11.) circle (4pt);
\draw [fill=black] (15.121320343559642,-5.878679656440358) circle (4pt);
\draw [fill=black] (15.77163859753386,-6.851949702904731) circle (4pt);
\draw [fill=black] (14.148050297095269,-5.22836140246614) circle (4pt);
\draw [fill=black] (11.851949702904731,-5.22836140246614) circle (4pt);
\draw [fill=ffqqqq] (10.878679656440358,-5.878679656440357) circle (5pt);
\draw [fill=black] (10.22836140246614,-6.85194970290473) circle (4pt);
\draw [fill=black] (10.22836140246614,-9.148050297095269) circle (4pt);
\draw [fill=black] (10.878679656440358,-10.121320343559642) circle (4pt);
\draw [fill=black] (11.851949702904731,-10.77163859753386) circle (4pt);
\draw [fill=black] (14.14805029709527,-10.77163859753386) circle (4pt);
\draw [fill=black] (15.77163859753386,-9.148050297095269) circle (4pt);
\draw [fill=black] (15.121320343559642,-10.121320343559642) circle (4pt);
\draw [fill=black] (12.029363136866388,-5.696119777435003) circle (1.5pt);
\draw [fill=black] (12.24704760647334,-5.6160824902940805) circle (1.5pt);
\draw [fill=black] (11.844523713937125,-5.783048364905913) circle (1.5pt);
\draw [fill=black] (10.689824769592988,-7.044442359240477) circle (1.5pt);
\draw [fill=black] (10.615333347887061,-7.249423582637672) circle (1.5pt);
\draw [fill=black] (10.781898019724252,-6.846733506470941) circle (1.5pt);
\draw [fill=qqqqff] (9.,-4.) circle (4pt);
\draw [fill=qqqqff] (8.503888888888886,-3.5038888888888855) circle (4pt);
\draw [fill=qqqqff] (9.503888888888886,-4.503888888888886) circle (4pt);
\draw [fill=ffqqqq] (-10.,-5.) circle (5pt);
\draw [fill=ffqqqq] (-4.,-5.) circle (5pt);
\end{tikzpicture}
\end{center}
\caption{Effect of an Edge in $H$ on the Board $B_\Gamma$}
\label{fig:edgetopath}
\end{figure}
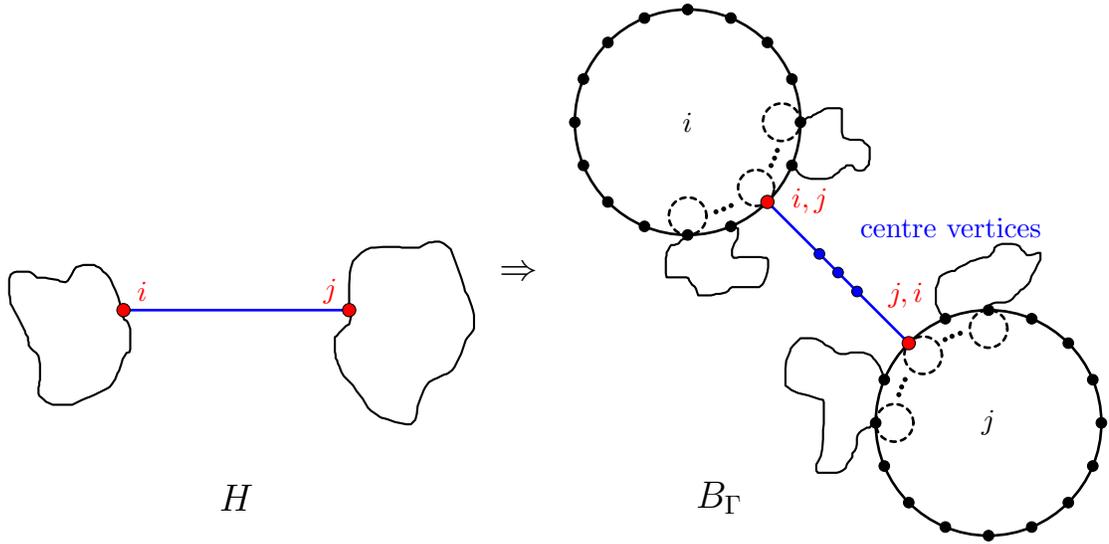
\end{construction}

As an example for this construction, consider the following:
\begin{example}\label{ex:graph1}
Let $\Gamma$ be a path of three vertices so that $H=\Gamma$. Let the two end vertices belong to $\L$, and the centre vertex to $\R$. Since $\Gamma$ consists of three vertices, \ie $n=3$, the cycle $i$ (where $i\in\L$) is of length $3^4+4=85$ with two cycles of length $3^3=27$ joined to two adjacent vertices, and the cycle $j$ (where $j\in\R$) is of length 86 with two cycles of length 27 joined to two adjacent vertices.

Label the edge between vertex 1 (an end vertex) and vertex 2 (the centre vertex) as 1, and the edge between vertex 2 and vertex 3 (the other end vertex) as 2.

The board $B_\Gamma$ is given in \cref{fig:exampleP3}. Dashed, blue cycles consist of 85 vertices, and dotted, red cycles of 86 vertices, with the two labelled vertices adjacent in both cases. The smaller solid cycles consist of 27 vertices.

\begin{figure}[!ht]
\begin{center}
\definecolor{ffqqqq}{rgb}{1.,0.,0.}
\definecolor{qqqqff}{rgb}{0.,0.,1.}
\begin{tikzpicture}[line cap=round,line join=round,>=triangle 45,x=1.0cm,y=1.0cm,scale=0.68]
\draw [line width=1.2pt,dash pattern=on 5pt off 5pt,color=qqqqff] (2.,2.) circle (2.cm);
\draw [line width=1.2pt,dotted,color=ffqqqq] (10.,2.) circle (2.cm);
\draw [line width=1.2pt,dash pattern=on 5pt off 5pt,color=qqqqff] (6.,-4.) circle (2.cm);
\draw [line width=1.2pt] (4.,2.)-- (8.,2.);
\draw [line width=1.2pt] (8.585786437626904,0.5857864376269051)-- (7.414213562373095,-2.585786437626905);
\draw [line width=1.2pt] (3.5,2.) circle (0.5cm);
\draw [line width=1.2pt] (3.0606601717798214,0.9393398282201788) circle (0.5cm);
\draw [line width=1.2pt] (8.5,2.) circle (0.5cm);
\draw [line width=1.2pt] (8.931149770827384,0.9475936205069617) circle (0.5001802370175656cm);
\draw [line width=1.2pt] (6.,-2.5) circle (0.5cm);
\draw [line width=1.2pt] (7.087054120990253,-2.966407557091235) circle (0.5019021186129954cm);
\draw [line width=1.2pt] (3.,7.)-- (6.,7.);
\draw [line width=1.2pt] (6.,7.)-- (9.,7.);
\draw (0,7.2) node {$H$};
\draw (0,-2.32) node {$B$};
\draw (2,2) node {$1$};
\draw (10,2) node {$2$};
\draw (6,-4) node {$3$};
\draw [fill=black] (4.,2.) circle (3pt);
\draw[color=black] (4.5,2.39) node {$1,2$};
\draw [fill=black] (3.414213562373095,0.5857864376269051) circle (3pt);
\draw[color=black] (3.9,0.25) node {$1,3$};
\draw [fill=black] (8.,2.) circle (3pt);
\draw[color=black] (7.5,2.39) node {$2,1$};
\draw [fill=black] (8.585786437626904,0.5857864376269051) circle (3pt);
\draw[color=black] (7.75,0.6) node {$2,3$};
\draw [fill=black] (7.414213562373095,-2.585786437626905) circle (3pt);
\draw[color=black] (8,-2.5) node {$3,2$};
\draw [fill=black] (6.,-2.) circle (3pt);
\draw[color=black] (6,-1.63) node {$3,1$};
\draw [fill=black] (6.,2.) circle (3pt);
\draw [fill=black] (7.9,-1.25) circle (3pt);
\draw [fill=black] (8.1,-0.75) circle (3pt);
\draw [fill=qqqqff] (3.,7.) circle (3pt);
\draw[color=qqqqff] (3,6.5) node {$1$};
\draw [fill=ffqqqq] (6.,7.) circle (3pt);
\draw[color=ffqqqq] (6,6.5) node {$2$};
\draw [fill=qqqqff] (9.,7.) circle (3pt);
\draw[color=qqqqff] (9,6.5) node {$3$};
\end{tikzpicture}
\end{center}
\caption{Constructing $B_{P_3}$}
\label{fig:exampleP3}
\end{figure}
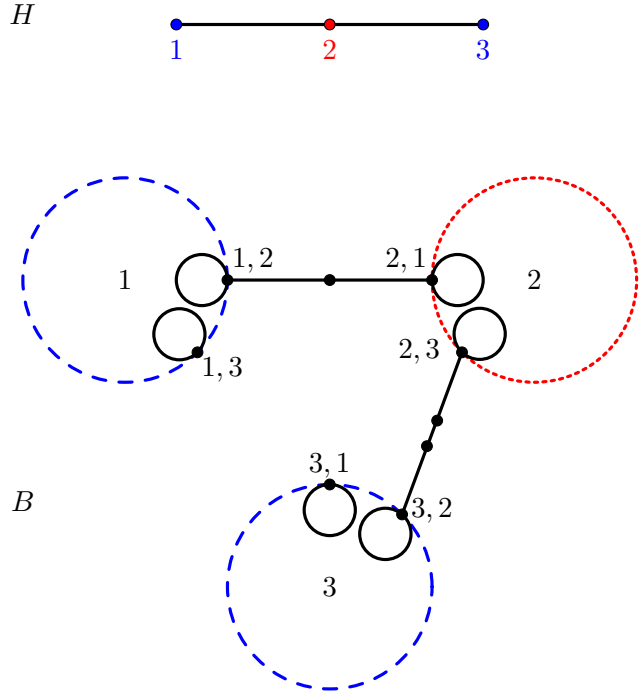
\end{example}

For the next construction, we will have to specify what is meant by distance between pieces.
\begin{definition}
Let two pieces $P_1$ and $P_2$ be placed on a board $B$ and let $V_1$ and $V_2$ be the set of vertices on which $P_1$, respectively $P_2$, was placed. We then define the \textbf{distance} $d(P_1, P_2)$ between $P_1$ and $P_2$ by
\[d(P_1,P_2)=\min\{d(v_1,v_2): v_1\in V_1, v_2\in V_2\},\]
where $d(v_1,v_2)$ is the graph theoretic distance between $v_1$ and $v_2$, \ie the minimum number of edges of a path in $B$ with endpoints $v_1$ and $v_2$.
\end{definition}

\begin{construction}[\textbf{$\Gamma$-ruleset}]
Given an $(\L,\R)$-labelled simplicial complex $\Gamma$ with no isolated vertices we construct a ruleset $R_\Gamma$ for an SP-game (called the \textbf{$\Gamma$-ruleset}).

If $\Gamma$ is empty, then let $R_\Gamma$ be the ruleset in which Left and Right place pieces on a single vertex with no restrictions.

If $\Gamma$ is non-empty, then construct $R_\Gamma$ as follows:
\begin{itemize}
\item Let $n$ be the number of vertices of $\Gamma$. Label the edges (the 1-dimensional faces) of $\Gamma$ as $\{1,\ldots,k\}$.
\item Left plays cycles of length $n^4+4$ with cycles of length $n^3$ joined to $n-1$ consecutive vertices,
\item Right plays cycles of length $n^4+5$ with cycles of length $n^3$ joined to $n-1$ consecutive vertices (\ie the pieces are as the structure given in \cref{fig:cyclei}), and
\item Let $F$ be a facet of $\Gamma$ of dimension $f-1$, whose 1-dimensional faces are labelled $k_1,\ldots, k_l$, where $l=\binom{f}{2}$. We call the set $\{k_1+1,\ldots, k_l+1\}$ the \textbf{id-set} of $F$. Then no sets of $f$ pieces are allowed such that the set of distances between any two pieces is exactly the id-set of $F$.
\end{itemize}
\end{construction}

\begin{example}\label{ex:graph2}
Let $\Gamma$ be a path of three vertices so that $n=3$. Left's pieces are cycles of length $3^4+4=85$ with two cycles of length $3^3=27$ joined to two adjacent vertices, and Rights pieces are cycles of length 86 with two cycles of length 27 joined to two adjacent vertices.

Since the facets of $\Gamma$ are the two edges (thus of size 2), the edge in one facet are labelled as $1$, and in the other as $2$. Thus the id-sets are $\{2\}$ and $\{3\}$, implying that in $R_\Gamma$ no two pieces are allowed to have distance 2 or distance 3.
\end{example}

\begin{example}
Consider $\Gamma=\langle abc, ad\rangle$. Label the edge between $a$ and $b$ as 1, between $b$ and $c$ as 2, between $c$ and $a$ as 3, and between $a$ and $d$ as 4.

For the facet $abc$ we have the id-set $\{1+1, 2+1, 3+1\}=\{2,3,4\}$. Thus in the $\Gamma$-ruleset $R_\Gamma$ we cannot have three pieces where the distances between pairs are $\{2,3,4\}$, while two with any one of these distance are allowed.

For the facet $ad$ we have the id-set $\{4+1\}=\{5\}$. Thus in $R_\Gamma$ we cannot have any two pieces with distance 5.
\end{example}

\begin{lemma}\label{lem:illegal}
Given an $(\L,\R)$-labelled simplicial complex $\Gamma$ with no isolated vertices, the $\Gamma$-ruleset $R_\Gamma$ is invariant.
\end{lemma}
\begin{proof}
If $\Gamma$ is empty, then $R_\Gamma$ played on any board has no illegal positions, thus is trivially invariant.

If $\Gamma$ is non-empty, then since $\Gamma$ has no isolated vertices, 
all facets have at least one edge and therefore all id-sets are non-empty. 
In particular, this means that every illegal position of $R_\Gamma$ played on any board has at least two pieces, so there are no illegal basic positions.

Now suppose that we are playing $R_\Gamma$ on isomorphic boards $B_1$ and $B_2$. Making a move to a position $P$ is legal on $B_1$ if and only if there is no id-set which is contained in the set of distances between pieces of $P$, which holds if and only if $P$ is legal on $B_2$.

Thus $R_\Gamma$ is invariant.
\end{proof}

The following statement will prove that every simplicial complex without isolated vertices can appear as the illegal complex of (many!) iSP-games.

\begin{theorem}[\textbf{Invariant Game from Illegal Complex}]\label{thm:Gammagameandgraph}
Given an $(\L,\R)$-labelled simplicial complex $\Gamma$ with no isolated vertices, fix labellings of the vertices and of the edges. Then $\Gamma$ is the illegal complex of the $\Gamma$-ruleset $R_\Gamma$ played on the $\Gamma$-board $B_\Gamma$, \ie $\Gamma_{R_\Gamma,B_\Gamma}=\Gamma$.
\end{theorem}
\begin{proof}
Let $G=(R,B)$ where $B=B_\Gamma$ and $R=R_\Gamma$ are the $\Gamma$-board and $\Gamma$-ruleset respectively, with the same labelling of the edges of $\Gamma$ if $\Gamma$ is nonempty.

If $\Gamma$ is empty, then $G$ has no illegal positions, thus $\Gamma_{R,B}$ is also empty.

To show that indeed $\Gamma_{R,B}=\Gamma$ for $\Gamma$ nonempty, we will begin by showing that their vertex sets have the same size.

Let $H=\Gamma^{[1]}$. Clearly Left can place one of her pieces on the cycle labelled $i$ in $B$ if the vertex $i$ of $H$ belongs to $\L$. Similarly Right can place on cycles labelled $j$ where $j\in\R$. Thus each vertex in $H$ corresponds to a position in the game $G$ played on $B$.

We now need to show that there are no other ways for Left or Right to place pieces than what was previously mentioned, \ie that the positions of $G$ correspond exactly to the vertices of $H$. 

Let $n$ be the number of vertices of $H$ and $k$ be the number of edges. The cycles in $B$ which only use connection and centre vertices have size at most $n(n-1)+\frac{k(k+1)}{2}$ (there are $n(n-1)$ connection vertices and $1+\ldots+k$ centre vertices). Since there are at most $\binom{n}{2}$ edges in $H$, we have
\begin{align*}
n(n+1)+\frac{k(k+1)}{2}&\le n(n+1)+\frac{\frac{n(n+1)}{2}\left(\frac{n(n+1)}{2}+1\right)}{2}\\
&=\frac{1}{8}n^4+\frac{1}{4}n^3+\frac{11}{8}n^2+\frac{5}{4}n
\end{align*}
which is less than $n^4+4$ for all whole numbers. 


Thus such cycles are shorter than $n^4+4$, and Left and Right will not be able to play on those.

Furthermore, any cycle of length $n^4+4$ or $n^4+5$ in $B$ needs to include the outer vertices of some cycle $i$ (since as above cycles using only connection and centre vertices are shorter, and the inner cycles are shorter). To then construct a cycle of that length without using all connection vertices of cycle $i$, the cycle would have to include at least one centre vertex. Since centre vertices do not have cycles of length $n^3$ added, this implies that neither Left or Right could play there.

Thus Left and Right are only able to play on the labelled cycles.

Further, since the pieces consist of cycles with a differing number of vertices, either player will only be able to play on the cycles of $B$ that are designated to them. Thus there are $n$ positions, in each of which only one player can play, all corresponding to vertices of $\Gamma$. The vertices of $\Gamma_{R,B}$ are thus a subset of the vertices of $\Gamma$ and $\Gamma_{R,B}$ has less vertices than $\Gamma$ if and only if there exists at least one position in which it is never illegal to play, which we will show cannot happen as part of the rest of the proof.

It remains to show that the facets of $\Gamma_{R,B}$ and $\Gamma$ correspond.

Consider a facet consisting of the vertices $i_1,\ldots, i_k$ in $\Gamma$, thus any two vertices have an edge between them in $H$, and let these edges be $j_1,\ldots, j_l$. Then the positions $i_a$ and $i_b$, $a,b\in \{1,\ldots, k\}$, in $B$ have distance $j_c+1$, where $j_c$ is the edge between $i_a$ and $i_b$ in $H$, (since we joined a path of length $j_c+2$ to their connection vertices). Thus it is illegal to play in all $k$ positions (and this is a minimal illegal position), and thus there is a facet consisting of the vertices $i_1,\ldots, i_k$ in $\Gamma_{R,B}$.

Now let the vertices $i_1,\ldots, i_k$ form a facet in $\Gamma_{R,B}$. Assume that $i_1,\ldots, i_k$ do not form a facet in $\Gamma$. If some subset $S$ of these vertices forms a facet, then by construction of $G$ it would be illegal to play pieces on all of the cycles in $B$ corresponding to vertices in $S$. Thus $i_1,\ldots, i_k$ is not a \textit{minimal} illegal position, a contradiction to those vertices forming a facet in $\Gamma_{R,B}$. If on the other hand $i_1,\ldots, i_k$ is strictly contained in some facet $F$ of $\Gamma$, then by construction of $G$ it is legal to play on cycles $i_1,\ldots, i_k$ in $B$. Thus $i_1,\ldots, i_k$ is not an \textit{illegal} position, a contradiction to those vertices forming a facet in $\Gamma_{R,B}$. Therefore $i_1,\ldots, i_k$ is a facet of $\Gamma$.

Finally, since $H$ has no isolated vertices (by $\Gamma$ not having such), the vertex set of $\Gamma_{R,B}$ is a subset of the vertex set of $H$, \ie the vertex set of $\Gamma$. Since furthermore the facets of $\Gamma_{R,B}$ and $\Gamma$ correspond, we have that the vertex set of $\Gamma_{R,B}$ is equal to that of $\Gamma$.

Consequently, the simplicial complexes $\Gamma$ and $\Gamma_{R,B}$ have the same vertex and facet sets, which proves $\Gamma=\Gamma_{R,B}$.
\end{proof}

\begin{example}\label{ex:graph3}
Let $\Gamma$ be a path of three vertices. Let $B=B_\Gamma$ (see \cref{ex:graph1}) and $R=R_\Gamma$ (see \cref{ex:graph2}).

Then $\Gamma_{R,B}=\Gamma$.
\end{example}

Note: Simpler constructions with smaller cycles and pieces are often possible (as shown in the next example), but the above construction is guaranteed to work.

\begin{example}
Let $\Gamma$ be as in \cref{ex:graph3}. Let Left play cycles of length $3$, and Right cycles of length $4$. For the board $B'$ given in \cref{fig:exampleP3simple}, it is easy to check that $\Gamma_{R',B'}=\Gamma$, where $R'$ is the ruleset which forbids overlap between pieces.

\begin{figure}[!ht]
\begin{center}
\begin{tikzpicture}[line cap=round,line join=round,>=triangle 45,x=1.0cm,y=1.0cm,scale=0.5]
\draw [line width=1.2pt] (3.,3.)-- (6.,3.);
\draw [line width=1.2pt] (6.,3.)-- (4.5,5.598076211353316);
\draw [line width=1.2pt] (4.5,5.598076211353316)-- (3.,3.);
\draw [line width=1.2pt] (3.,3.)-- (0.,3.);
\draw [line width=1.2pt] (0.,3.)-- (0.,0.);
\draw [line width=1.2pt] (0.,0.)-- (3.,0.);
\draw [line width=1.2pt] (3.,0.)-- (3.,3.);
\draw [line width=1.2pt] (6.,3.)-- (6.,0.);
\draw [line width=1.2pt] (6.,0.)-- (9.,0.);
\draw [line width=1.2pt] (9.,0.)-- (9.,3.);
\draw [line width=1.2pt] (9.,3.)-- (6.,3.);
\draw (1.5,1.5) node {$1$};
\draw (4.5,4.25) node {$2$};
\draw (7.5,1.5) node {$3$};
\begin{scriptsize}
\draw [fill=black] (0.,3.) circle (5pt);
\draw [fill=black] (3.,3.) circle (5pt);
\draw [fill=black] (3.,0.) circle (5pt);
\draw [fill=black] (0.,0.) circle (5pt);
\draw [fill=black] (6.,3.) circle (5pt);
\draw [fill=black] (6.,0.) circle (5pt);
\draw [fill=black] (9.,0.) circle (5pt);
\draw [fill=black] (9.,3.) circle (5pt);
\draw [fill=black] (4.5,5.598076211353316) circle (5pt);
\end{scriptsize}
\end{tikzpicture}
\end{center}
\caption{Smaller board $B'$}
\label{fig:exampleP3simple}
\end{figure}
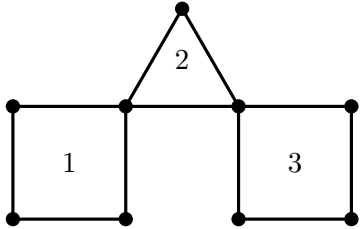
\end{example}

The following theorem summarizes our results about illegal complexes of iSP-games.

\begin{theorem}\label{thm:illegalcomplexinvariant}
A given simplicial complex $\Gamma$ is the illegal complex of some iSP-game $(R,B)$ if and only if $\Gamma$ has no isolated vertices.
\end{theorem}
\begin{proof}
By \cref{thm:illegalinvariantnovertices} we have that if $\Gamma$ is the illegal complex of an iSP-game, then $\Gamma$ has no isolated vertices.

Conversely, if $\Gamma$ has no isolated vertices, then by \cref{thm:Gammagameandgraph}, we have that $\Gamma$ is the illegal complex of some iSP-game.
\end{proof}

We will now consider legal complexes. The first result shows that \textit{every} simplicial complex is the legal complex of some iSP-game and board.
\begin{theorem}[\textbf{Invariant Game from Legal Complex}]\label{thm:legalinv}
Given any $(\L,\R)$-labelled simplicial complex $\Delta$, we can construct an iSP-game $(R,B)$ such that $\Delta=\Delta_{R,B}$ and the sets of Left, respectively Right, positions is $\L$, respectively $\R$.
\end{theorem}
\begin{proof}
We will take advantage of the disjunctive sum of two SP-games corresponding to the join of their legal complexes (see \cref{thm:disjunctiveStructure}).

Let $U=\{v_1,\ldots,v_\ell\}$ be the set of vertices contained in all facets of $\Delta$ and let $\Delta'=\langle F_1\setminus U,\ldots, F_k\setminus U\rangle$ where $F_1,\ldots F_k$ are the facets of $\Delta$. Note that $U$ could be empty or be the entire vertex set. Then $\Delta=\Delta' *\langle U\rangle$.

We will construct iSP-games $(R_1,B_1)$ and $(R_2,B_2)$ such that $\Delta_{R_1,B_1}=\Delta'$ and $\Delta_{R_2,B_2}=\langle U\rangle$. Then $(R_1,B_1)+(R_2,B_2)$ is an iSP-game with legal complex $\Delta$.

First, given $\Delta'$ let $\Gamma'=\facet{\SR{\Delta'}}$, \ie the simplicial complex whose facets correspond to the minimal non-faces of $\Delta'$.

Let the vertex set of $\Gamma'$ be bipartitioned into $\L$ and $\R$ the same way that the vertex set of $\Delta'$ is. Let $R_1$ be the $\Gamma'$-ruleset and $B_1$ be the $\Gamma'$-board, so that $\Gamma_{R_1,B_1}=\Gamma'$. 

Let $i$ be a vertex in $\Delta'$. Since $\Delta'$ by construction has at least one facet that does not contain $i$ we have that $i$ is also a vertex of $\Gamma'$. Thus the underlying rings of $\Gamma'$ and $\Gamma_{R_1,B_1}$ are the same, and it immediately follows that \[\Delta_{R_1,B_1}=\SR{\facet{\Gamma_{R_1,B_1}}}=\SR{\facet{\Gamma'}}=\Delta'.\]

Secondly, we construct $R_2$ and $B_2$ as follows: Let $n$ be the number of vertices in $\langle U\rangle$ and (re)label the vertices $1,\ldots, n$. Let the board $B_2$ be a disjoint union of $n$ cycles of size $3$ and $4$ and label these $1,\ldots, n$ so that cycle $i$ will have size $3$ if the vertex $i$ in $\langle U\rangle$ belongs to $\L$, and size $4$ if the vertex $i$ belongs to $\R$.

Let $R_2$ be the SP-ruleset in which Left plays cycles of length $3$, and Right plays cycles of length $4$. Note that $R_2$ is invariant.

It is easy to see that $\langle U\rangle=\Delta_{R_2,B_2}$.
\end{proof}

In the above proof, we constructed the iSP-game as a disjunctive sum of two iSP-games. If it is desired for some reason that a single ruleset and board are constructed, an alternative proof can be found in the PhD thesis of the second author \cite{HuntemannPhD}.

The following example demonstrates this construction.

\begin{example}
Consider the complex $\Delta=\langle ab, bc\rangle$, where the vertices are partitioned as $\L=\{a,b\}$ and $\R=\{c\}$. In this case $U=\{b\}$ and $\Delta'=\langle a,c\rangle$.

For $\Delta'$ we have $\Gamma'=\langle ac\rangle$, thus the graph $H$ is $P_2$. Since $n=2$, in the SP-ruleset $R_1$ Left will play cycles of length $n^4+4=20$ with one cycle of length $n^3=8$ added to a vertex, while Right plays cycles of length $n^4+5=21$ with a cycle of length 8 added to a vertex, and pieces may not have distance 2.

The board $B_1$ is given in the top half of \cref{fig:exampleFlag3}. Dashed, blue cycles consist of 20 vertices, and dotted, red cycles of 21 vertices. The smaller solid cycles consist of 8 vertices. It is now easy to check that $\Delta_{R_1,B_1}=\Delta$.

For $\langle U\rangle$, the ruleset $R_2$ is that Left may only play cycles of length 3, and Right only cycles of length 4, and the board is a cycle of length 3 (bottom half of \cref{fig:exampleFlag3}).

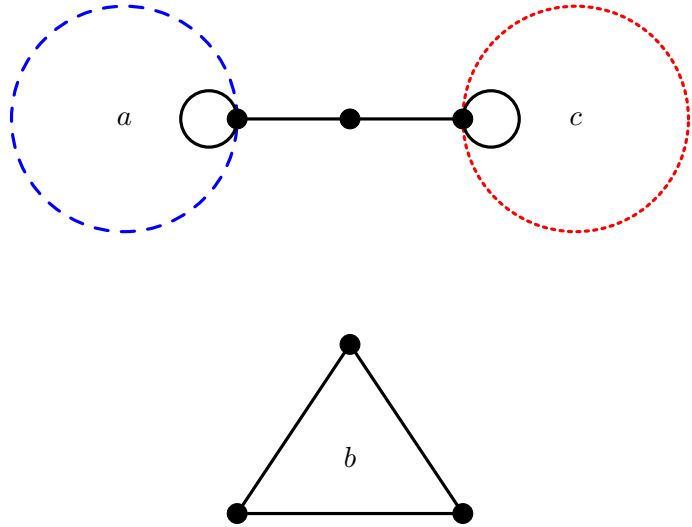
\begin{figure}[!ht]
\begin{center}
\definecolor{ffqqqq}{rgb}{1.,0.,0.}
\definecolor{qqqqff}{rgb}{0.,0.,1.}
\begin{tikzpicture}[line cap=round,line join=round,>=triangle 45,x=1.0cm,y=1.0cm,scale=0.75]
\draw [line width=1.2pt,dash pattern=on 5pt off 5pt,color=qqqqff] (2.,2.) circle (2.cm);
\draw [line width=1.2pt,dotted,color=ffqqqq] (10.,2.) circle (2.cm);
\draw [line width=1.2pt] (4.,2.)-- (8.,2.);
\draw [line width=1.2pt] (3.5,2.) circle (0.5cm);
\draw [line width=1.2pt] (8.5,2.) circle (0.5cm);
\draw[line width=1.2pt] (6,-2)--(8,-5)--(4,-5)--cycle;
\draw (2,2) node {$a$};
\draw (10,2) node {$c$};
\draw (6,-4) node {$b$};
\begin{scriptsize}
\draw [fill=black] (4.,2.) circle (5pt);
\draw [fill=black] (8.,2.) circle (5pt);
\draw [fill=black] (6.,-2.) circle (5pt);
\draw [fill=black] (6.,2.) circle (5pt);
\draw [fill=black] (8,-5) circle (5pt);
\draw [fill=black] (4,-5) circle (5pt);
\end{scriptsize}
\end{tikzpicture}
\end{center}
\caption{Constructing $B$ from $\Delta=\langle ab, bc\rangle$}
\label{fig:exampleFlag3}
\end{figure}

The sum $(R_1,B_1)+(R_2,B_2)$ is then the game $(R,B)$ where in $R$ Left may play cycles of length 3 or cycles of length 20 with an added cycle of length 8, Right cycles of length 4 or cycles of length 21 with an added cycle of length 8, and no two pieces may have distance 2, and the board $B$ is the disjoint union of the boards $B_1$ and $B_2$.

Then $\Delta=\Delta_{R,B}$.
\end{example}

Concluding our discussion of iSP-games, we have the following result.

\begin{theorem}[\textbf{Every SP-Game Tree Belongs To An iSP-Game}]
Given an SP-game $(R,B)$, there exists an iSP-game $(R',B')$ so that their game trees are isomorphic.
\end{theorem}
\begin{proof}
Let $\Delta=\Delta_{R,B}$ with $\L$ the vertices corresponding to Left basic positions, and similarly $\R$. Then by \cref{thm:legalinv} we know that there exists an iSP-game $(R',B')$ such that $\Delta=\Delta_{R',B'}$ with the same bipartition. Since $\Delta_{R,B}=\Delta_{R',B'}$, we have by \cref{thm:isomorphicgametrees} that the game trees of $R$ played on $B$ and $R'$ played on $B'$ are isomorphic.
\end{proof}

This in particular implies that under most winning conditions (such as normal play or mis\`ere play) the game values of $R$ played on $B$ and $R'$ played on $B'$ are the same, implying that we can replace one by the other.

\section{Independence Games}\label{sec:independence}

Many of the games we have previously considered have illegal complexes that are graphs. This special class of SP-games is of further interest to us. For example, this class corresponds to flag complexes (see below for more).

\begin{definition}
An SP-ruleset $R$ is called an \textbf{independence ruleset} if for \textit{any} board $B$ the illegal complex $\Gamma_{R,B}$ is a graph without isolated vertices (\ie a pure one-dimensional simplicial complex). An SP-game $(R,B)$ is called an \textbf{independence game} if $R$ is an independence ruleset.
\end{definition}

Consider the illegal complex $\Gamma_{R,B}$ of an independence game $R$ on a board $B$. Let $\Gamma_{R,B}'$ be the graph on the vertex set $x_1,x_2,\ldots,x_m, y_1, y_2,\ldots, y_n$ (corresponding to the basic positions of $R$ played on $B$) with edges those of $\Gamma_{R,B}$. Thus the difference between $\Gamma_{R,B}$ and $\Gamma_{R,B}'$ are isolated vertices corresponding to basic positions that are always legal. For many independence games we have $\Gamma_{R,B}'=\Gamma_{R,B}$.

The \textbf{independence complex} of a graph $H$ is a simplicial complex with vertex set that of the graph and faces those sets of vertices that are independent in $H$, \ie no two vertices are adjacent. The term `independence game' was chosen for this class of games since the independent sets of $\Gamma_{R,B}'$ correspond to the legal positions of $R$ played on $B$, \ie the faces of $\Delta_{R,B}$. Thus in this case $\Delta_{R,B}$ is the independence complex of the graph $\Gamma_{R,B}'$.

Many SP-games, such as \col and \Snort, are independence games. \nogo is an example of an SP-game that is not an independence game. Even though $\Gamma_{\Nogo, B}$ is a graph for some boards (for example when $B$ is the graph on two vertices connected by an edge, \ie the path of length one, $P_2$), there are many others for which this is not the case. For example, $\Gamma_{\Nogo, P_3}$, given in \cref{fig:NogoP3aux}, has two-dimensional faces.

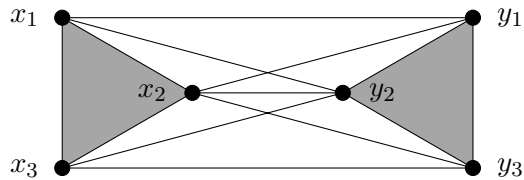
\begin{figure}[!ht]
\begin{center}
\begin{tikzpicture}
	\filldraw[fill=gray!70, draw=black] (0, 1)--(1.73, 0)--(0, -1)-- cycle;
	\filldraw[fill=gray!70, draw=black] (5.46, 1)--(3.73,0)--(5.46, -1)-- cycle;
	\draw (0,1)--(5.46,1)--(1.73,0)--(5.46,-1)--(0,-1)--(3.73,0)--(0,1);
	\draw (1.73,0)--(3.73,0);
	\filldraw (0,1) circle (0.1cm);
	\filldraw (1.73,0) circle (0.1cm);
	\filldraw (0,-1) circle (0.1cm);
	\filldraw (5.46,1) circle (0.1cm);
	\filldraw (3.73,0) circle (0.1cm);
	\filldraw (5.46,-1) circle (0.1cm);
	\draw (-0.5,1) node {$x_1$};
	\draw (1.2,0) node {$x_2$};
	\draw (-0.5,-1) node {$x_3$};
	\draw (5.96,1) node {$y_1$};
	\draw (4.26,0) node {$y_2$};
	\draw (5.96,-1) node {$y_3$};
\end{tikzpicture}
\end{center}
\caption{The Illegal Complex $\Gamma_{\Nogo, P_3}$}
\label{fig:NogoP3aux}
\end{figure}

One nice property of independence games is that playing an independence game $R$ on a board $B$ is equivalent to forming independent sets of the graph $\Gamma_{R,B}'$ while Left picks vertices in $\L$ and Right in $\R$.

A \textbf{flag complex} $\Delta$ is a complex whose minimal non-faces all have size 2 (see for example \cite{HH11}). In the case of independence games, since $\Gamma_{R,B}$ is a graph without isolated vertices, we have that $\Delta_{R,B}$ is flag.

Further note that the $\Gamma$-ruleset in the case of $\Gamma$ being a graph is always an independence ruleset (since minimal illegal positions are always pairs of pieces played). Using \cref{thm:legalinv} this implies the following.

\begin{proposition}[\textbf{iSP-Games of Flag Complexes}]\label{cor:independenceSP}
Given any SP-game $(R,B)$ such that $\Gamma_{R, B}$ is a non-empty graph, there exists an invariant independence game $(R',B')$ such that $\Delta_{R,B}= \Delta_{R', B'}$. In the case that $\Gamma_{R,B}$ has no isolated vertices, we also have $\Gamma_{R,B}=\Gamma_{R',B'}$.
\end{proposition}
\begin{proof}
By \cref{thm:legalinv} there exists an iSP-ruleset $R'$ and board $B'$ such that $\Delta_{R,B}=\Delta_{R',B'}$. The ruleset $R_1'$ is the $\Gamma$-ruleset of $\Gamma_{R,B}'$ which, as mentioned above, is an independence ruleset. The ruleset $R_2'$ has no illegal positions, and thus is an independence ruleset trivially.

If $\Gamma_{R,B}$ has no isolated vertices, then the underlying rings of $\Delta_{R,B}$ and $\Delta_{R',B'}$ are the same, thus
\[\Gamma_{R',B'}=\facet{\SR{\Delta_{R',B'}}}=\facet{\SR{\Delta_{R,B}}}=\Gamma_{R,B}.\qedhere\]
\end{proof}

Equivalently, this proposition also states that given an SP-ruleset $R$ and board $B$ such that the minimal non-faces of $\Delta_{R,B}$ are all 1- and 2-element sets, there exists an SP-ruleset $R'$ whose legal complex is always flag and a board $B'$ such that $\Delta_{R,B}=\Delta_{R',B'}$.

As a direct consequence of \cref{cor:independenceSP}, applying \cref{thm:isomorphicgametrees}, we have that these games also have isomorphic game trees.

\begin{corollary}
Given any SP-game $(R,B)$ such that $\Gamma_{R, B}$ is a non-empty graph, there exists an invariant independence game $(R',B')$ such the game trees of $(R,B)$ and $(R',B')$ are isomorphic.
\end{corollary}

\section{Further Questions and Work}\label{sec:discussion}

In this section, we will be discussing some potential further questions and avenues to explore.

The $\Gamma$-board and pieces of the $\Gamma$-rulesey have many more vertices than $\Gamma$ itself. Thus we are interested in whether constructions of a ruleset $R$ and board $B$ are possible for every simplicial complex $\Gamma$ without isolated vertices in which the pieces that Left and Right play occupy only one vertex so that $\Gamma=\Gamma_{R,B}$. This seems unlikely though, thus an interesting question is for which class of simplicial complexes such a construction is possible.

Similarly, we are also interested in for which simplicial complexes $\Delta$ we can find a ruleset $R$ and board $B$ with pieces only a single vertex so that $\Delta=\Delta_{R,B}$.

Simplicial trees and forests, which are generalizations of graph trees and forests, are flag complexes (see \cite[Lemma 9.2.7]{HH11}). Since many properties of simplicial trees are known (see for example \cite{Faridi04} and \cite{Faridi05}) it seems that this class of flag complexes provides a good start to studying whether simpler constructions are possible.

Finally, it is of interest if each game value possible under normal play conditions is also the game value of some SP-game. This problem has received attention for specific SP-rulesets (for \textsc{Domineering} see for example \cite{Kim96,UB15}, for \col and \snort see \cite{BCG04}), and was recently positively answered for a non-SP-game (see \cite{CS16}). Since SP-games are much easier to understand than many other combinatorial games, if the answer to this question is positive, it would provide an excellent new tool for studying combinatorial games. Whether or not this is the case, a similar, but stronger, question is if the simplest game (essentially the one with smallest game tree) in each equivalence class containing an SP-game is itself an SP-game. Knowing that each simplicial complex is the legal complex of some SP-game has been indispensable in the exploration of those two questions (see \cite{Huntemann17}).

\bibliographystyle{plain}
\bibliography{Biblio2018Oct}

\end{document}